\documentclass[10pt,leqno]{amsart}
\usepackage{txfonts}
\usepackage{amsmath,amsthm,amsfonts}
\usepackage{latexsym}
\usepackage{amssymb}
\usepackage[dvips]{epsfig}

\renewcommand{\thesection}{\arabic{section}}

\newtheorem{theorem}{Theorem}[section]
\newtheorem{lemma}[theorem]{Lemma}
\newtheorem{prop}[theorem]{Proposition}

\newtheorem{corollary}[theorem]{Corollary}
\theoremstyle{remark}
\newtheorem{remark}[theorem]{Remark}

\renewcommand{\theequation}{\thesection .\arabic{equation}}
\let\subs\subsection
\renewcommand\subsection{\setcounter{equation}{0}
\gdef\theequation{\thesubsection \arabic{equation}}\subs}
\let\sect\section
\renewcommand\section{\setcounter{equation}{0}
\gdef\theequation{\thesection .\arabic{equation}}\sect}

\makeatletter

\newcommand{\Rmnum}[1]{\expandafter\@slowromancap\romannumeral #1@}
\makeatother

\newcommand{\IC}{{\mathbb{C}}}

\newcommand{\IR}{{\mathbb{R}}}
\newcommand{\TT}{{\mathbb{T}}}

\newcommand{\tor}{\TT}

\newcommand{\be}{\begin{eqnarray}}
\newcommand{\ee}{\end{eqnarray}}

\newcommand{\mes}{\mathop{\rm{mes}\, }}

\def\beeq{\begin{equation}}
\def\eneq{\end{equation}}

\def\bm{\begin{matrix}}
\def\endm{\end{matrix}}

\begin{document}

\title[H\"older continuity of Lyapunov exponent for quasi-periodic
Jacobi operators]{H\"older continuity of Lyapunov exponent for
quasi-periodic Jacobi operators}

\author{Kai Tao}

\address{Department of Mathematics, Nanjing University, 22 Hankou Road Nanjing Jiangsu 210093 P.R.China }

\email{tao.nju@gmail.com}

\thanks{This work was done during 2010-2011 year when author was a visiting doctoral student at the Department of Mathematics
University of Toronto. The author wishes to thank M.Goldstein  for
his supervision of the project. }

\date{}

\begin{abstract}We consider the quasi-periodic Jacobi operator $H_{x,\omega}$ in $l^2(\mathbb{Z})$
\[
(H_{x,\omega}\phi)(n)=-b(x+(n+1)\omega)\phi(n+1)-b(x+n\omega)\phi(n-1)+a(x+n\omega)\phi(n)=E\phi(n)
,\ n\in\mathbb{Z},\]where $a(x),\ b(x)$ are analytic function on
$\mathbb{T}$ , $b$ is not identically zero, and $\omega$ obeys some strong Diophantine condition.
 We consider the corresponding unimodular cocycle. We prove that if the Lyapunov exponent $L(E)$ of the cocycle
 is positive for some $E=E_0$, then there exists $\rho_0=\rho_0(a,b,\omega,E_0)$, $\beta=\beta(a,b,\omega)$
such that $|L(E)-L(E')|<|E-E'|^\beta$ for any $E,E'\in (E_0-\rho_0,E_0+\rho_0)$. If $L(E)>0$ for all $E$ in some compact interval
$I$ then $L(E)$ is H\"{o}lder continuous on $I$ with a H\"{o}lder exponent $\beta=\beta(a,b,\omega,I)$. In our derivation we follow
the refined version of the Goldstein-Schlag method ~\cite{GS} developed by Bourgain and Jitomirskaya ~\cite{BJ}.
\end{abstract}

 \maketitle

\section{Introduction}
We consider the quasi-periodic Jacobi operator $H_{x,\omega}$ in $l^2(\mathbb{Z})$
\[
(H_{x,\omega}\phi)(n)=-b(x+(n+1)\omega)\phi(n+1)-b(x+n\omega)\phi(n-1)+a(x+n\omega)\phi(n)=E\phi(n)
,\ n\in\mathbb{Z},\]where $a(x),\ b(x)$ are analytic function on
$\mathbb{T}$, $b$ is not identically zero. Set
\[
A(x,E,\omega)=\frac{1}{b(x+\omega)}\left( \begin{array}{cc}
a(x)-E& -b(x) \\
b(x+\omega)& 0
\end{array}\right )
.\]
\[
M_{N}(x,E,\omega)=M_{[1,N]}(x,E,\omega)=A(x+(N-1)\omega,E,\omega)A(x+(N-2)\omega,E,\omega)\cdots
A(x,E,\omega),
\]
Define the unimodular matrix
\[
\tilde{M}_{N}(x,E,\omega)=\tilde{M}_{[1,N]}(x,E,\omega):=\frac{M_{[1,N]}(x,E,\omega)}{|\det
M_{[1,N]}(x,E,\omega)|^{\frac{1}{2}}}.
\]
As
\[
\det A(x,E,\omega)= \frac{b(x)}{b(x+\omega)},
\]then
\begin{equation}\label{10001}
\det M_{[1,N]}(x,E,\omega)=
\prod_{n=0}^{N-1}\frac{b(x+n\omega)}{b(x+(n+1)\omega)}
=\frac{b(x)}{b(x+N\omega)},
\end{equation}and
\be\label{10002}
\log\|\tilde{M}_{[1,N]}(x,E,\omega)\|=\log\|M_{[1,N]}(x,E,\omega)\|-\frac{1}{2}\log|\frac{b(x)}{b(x+N\omega)}|.\ee
\begin{remark}\label{13002}
\begin{enumerate}
\item[{\rm{(1)}}]Note that
\[\|A(x,E,\omega)\|\leq \frac{C(a,b,E)}{|b(x+\omega)|},\]
where the constants $C(a,b,E)$ obeys $C(a,b,E_0)=\sup_{|E|\leq
E_0}C(a,b,E)<+\infty.$ Therefore
\[\frac{1}{N}\log\|M_{[1,N]}(x,E,\omega)\|\leq \log C(a, b,
E)-\frac{1}{N}\sum_{n=1}^N\log|b(x+n\omega)|.\] In this paper we
always assume that $|E|\leq E_0$, where $E_0$ depends on $a,b$. For
that matter we suppress $E$ from the notations of some of the constants
involved.

\item[{\rm{(2)}}] $\log\|\tilde{M}_{[1,N]}(x,E,\omega)\|\geq 0$,
since $\tilde{M}_{[1,N]}(x,E,\omega)$ is unimodular.

\item[{\rm{(3)}}]\begin{eqnarray}
0&\leq &\frac{1}{N}\log
\|\tilde{M}_{[1,N]}(x,E,\omega)\|=\frac{1}{N}\log
\|M_{[1,N]}(x,E,\omega)\|-\frac{1}{2N}\log|\frac{b(x)}{b(x+N\omega)}|\nonumber\\
&\leq &\log
C(a,b,E)-\frac{1}{N}\sum_{n=1}^N\log|b(x+n\omega)|-\frac{1}{2N}\log|\frac{b(x)}{b(x+N\omega)}|\nonumber
\end{eqnarray}
\item[{\rm{(4)}}]It is well-known fact that if $b$ is analytic function which is not identically zero then $(\log|b|)^2$ is
integrable. Set
\[D=\int_{\mathbb{T}}\log |b(\theta)|d\theta .\]
Therefore
\[\int_{\mathbb{T}} \left |\frac{1}{N}\log
\|\tilde{M}_{[1,N]}(x,E,\omega)\| \right |dx =\int_{\mathbb{T}}
\frac{1}{N}\log \|\tilde{M}_{[1,N]}(x,E,\omega)\|dx \leq C'(a,
b)-D:=C''(a,b).\]Similarly
\[\int_{\mathbb{T}}(\frac{1}{N}\log\|\tilde{M}_N(x,E,\omega)\|)^2dx
\leq \tilde{C}(a,b).\]
\item[{\rm{(5)}}]Combining (4) with (\ref{10002}), we concludes that
$\frac{1}{N}\log\|M_{[1,N]}(x,E,\omega)\|$ is integrable, and
\[
 \frac{1}{N} \int_{\mathbb{T}} \log
\|M_N(x,E,\omega)\|  dx= \frac{1}{N} \int_{\mathbb{T}} \log
\|\tilde{M}_N(x,E,\omega)\|  dx .\]
\end{enumerate}
\end{remark}
~\\  ~\\ ~\\ Set \[L_N(E,\omega)=\frac{1}{N} \int_{\mathbb{T}} \log
\|M_N(x,E,\omega)\| dx =\frac{1}{N} \int_{\mathbb{T}} \log
\|\tilde{M}_N(x,E,\omega)\| dx .\] Note that $L_N(E,\omega)>0$. Set
\[
B(x,E,\omega):=\left( \begin{array}{cc}
a(x)-E& -b(x) \\
b(x+\omega)& 0
\end{array}\right ),
\]and
\[
T_{N}(x,E,\omega)=T_{[1,N]}(x,E,\omega):=B(x+(N-1)\omega,E,\omega)B(x+(N-2)\omega,E,\omega)\cdots
B(x,E,\omega).
\]
Then
\[M_{[1,N]}(x,E,\omega)=T_{[1,N]}(x,E,\omega)\prod_{n=N-1}^0\frac{1}{b(x+(n+1)\omega)},\]
\begin{eqnarray}\label{10003}
\tilde{M}_{[1,N]}(x,E,\omega)&=&\frac{|b(x+N\omega)|^\frac{1}{2}}{|b(x)|^\frac{1}{2}}M_{[1,N]}(x,E,\omega)
\\
&=&\frac{|b(x+N\omega)|^\frac{1}{2}}{|b(x)|^\frac{1}{2}}
\prod_{n=0}^{N-1}\frac{1}{b(x+(n+1)\omega)}
T_{[1,N]}(x,E,\omega)\nonumber,
\end{eqnarray}
\be\label{16001}\|\tilde{M}_{[1,N]}(x,E,\omega)\|
=\prod_{n=0}^{N-1}\frac{1}{|b(x+n\omega)b(x+(n+1)\omega)|^{\frac{1}{2}}}\|T_{[1,N]}(x,E,\omega)\|,\ee
and
\begin{equation}\label{10004}
\log\|\tilde{M}_{[1,N]}(x,E,\omega)\|=\log
\|T_{[1,N]}(x,E,\omega)\|-\frac{1}{2}\sum_{n=0}^{N-1}\log
|b(x+n\omega)b(x+(n+1)\omega)| .
\end{equation}Note also for the future references that \be \label{14001} |\det
T_{[1,N]}(x,E,\omega)|=\prod_{n=0}^{N-1}|b(x+n\omega)||b(x+(n+1)\omega)|.\ee
Combining (\ref{10004}) with Remark \ref{13002}, one concludes that
$\frac{1}{N}\log\|T_{[1,N]}(x,E,\omega)\|$ is integrable,
\be\label{10005} J_N(E,\omega):= \frac{1}{N} \int_{\mathbb{T}} \log
\|T_{[1,N]}(x,\omega )\|  dx =L_N(E,\omega)+D.\ee  Due to the
subadditive property, the limits
\begin{eqnarray}\label{10006}
L(E,\omega)&=& \lim\limits_{N\to \infty} \int_{\mathbb{T}}
\frac{1}{N} \log \| M_N(x,E,\omega ) \| dx=\lim\limits_{N\to \infty}
\int_{\mathbb{T}} \frac{1}{N} \log \| \tilde{M}_N(x,E,\omega ) \|
dx\\ &=&\lim\limits_{N\to \infty} L_N(E,\omega)\nonumber,
\end{eqnarray}
\begin{equation}\label{10007}
J(E,\omega)= \lim\limits_{N\to \infty} \int_{\mathbb{T}} \frac{1}{N}
\log \| T_N(x,E,\omega)\|  dx=\lim\limits_{N\to \infty}
J_N(E,\omega)=L(E,\omega)+D
\end{equation}
exist. Moreover, $L(E,\omega)\geq 0$. Fix some $\alpha>1$.
Throughout this paper we assume that $\omega\in(0,1)$ satisfies the
Diophantine condition
\begin{equation}\label{13001}
\|n\omega\| \geq \frac{C_{\omega}}{n(\log n)^\alpha}\ \ \mbox{for
all}\ n.\end{equation} It is well known that for a fixed $\alpha>1$
almost every $\omega$ satisfies (\ref{13001}).

 The main theorem in this paper is
\begin{theorem}\label{mainthm}
Assume $L(E_0)>0$. Then there exists $\rho_0>0$ depending on
$a(x),b(x),\omega$ and $E_0$ such that for any $E,E'\in (E_0-\rho_0,
E_0+\rho_0)$ holds
\[|L(E)-L(E')|=|J(E)-J(E')|< |E-E'|^{\beta},\] where
 the constant $\beta$ here depends on $a(x), b(x),\omega$, i.e.
$\beta=\beta(a,b,\omega)$, but does not depend on $E_0$.
\end{theorem}

\section{Large Deviation Theorem}

 It is convenient to replace
$a(x)$,$b(x)$ by $p(e( x))$ and $q(re(x))$ (with $e(x)=e^{2\pi
ix}$),where $p(z),q(z)$ are analytic function in the annulus
$A_\rho=\{z\in\mathbb{C}:1-\rho< |z|< 1+\rho\}$ which assume only
real values for $|z|=1$. With this convention in place, we will use
the notation $B(z)$, $T_{[1,N]}(z)$ e.t.c..

\begin{lemma}\label{lem12002}
\begin{enumerate}
\item[{\rm{(1)}}]$\sup_{z\in A_{\frac{\rho}{2}}} \| B(z)\|=\sup_{z\in A_{\frac{\rho}{2}}} \left \| \left( \begin{array}{cc}
p(z)-E& -q(z) \\
q(ze(\omega))& 0
\end{array}\right )\right \|\leq C$,where $C=C(p,q)$, and
$C(p,q)$ is the same as in (1) of Remark \ref{13002}.

\item[{\rm{(2)}}]$\sup_{z\in A_{\frac{\rho}{2}}}  \| B^{-1}(z)\|\leq
\frac{C(p,q)}{|q(ze(\omega))q(z))|}$
\end{enumerate}
\end{lemma}
\begin{proof}
(1) is obvious;\\

(2)$\det B(z)=q(ze(\omega))q(z),and\
B^{-1}(z)=\frac{1}{q(z)q(ze(\omega))}\left(
\begin{array}{cc}
0& q(z) \\
-q(ze(\omega))&p(z)-E
\end{array}\right )$, that implies $(2)$.\end{proof}
So
\begin{lemma}\label{12003}
For any $z\in A_{\frac{\rho}{2}}$,
\begin{equation}\label{14002}-C_1(p,q)+\log|q(z)q(ze(\omega))|\leq \log
\|T_{[1,N]}(z)\|-\log \|T_{[1,N]}(ze(\omega))\|\leq C_1(p,q)-\log
|q(ze((N-1)\omega))q(ze(N\omega))|.\end{equation}
\end{lemma}
\begin{proof}$T_{[1,N]}(ze(\omega))=B(ze(N\omega))T_{[1,N]}(z)B^{-1}(z)$,so
$\|T_{[1,N]}(ze(\omega))\|\leq C(p,q) \|T_{[1,N]}(z)\|
\frac{C(p,q)}{|q(z)q(ze(\omega))|}.$
\begin{displaymath}
\Rightarrow \log \|T_{[1,N]}(ze(\omega))\|\leq 2\log C(p,q) +\log
\|T_{[1,N]}(z)\|-\log |q(z)q(ze(\omega))| .\end{displaymath}
Similarly, we have
\begin{displaymath}
\log \|T_{[1,N]}(z)\| \leq 2\log C(p,q) +\log
\|T_{[1,N]}(ze(\omega))\|-\log |q(ze((N-1)\omega))q(ze(N\omega))|.
\end{displaymath}
Thus
\[-C_1+\log|q(z)q(ze(\omega))|\leq \log \|T_{[1,N]}(z)\|-\log
\|T_{[1,N]}(ze(\omega))\|\leq C_1-\log
|q(ze((N-1)\omega))q(ze(N\omega))|,\]where $C_1=C_1(p,q)$.
\end{proof}

\begin{corollary}\label{12004}
\begin{enumerate}
\item[{\rm{(1)}}]\begin{eqnarray}&&-kC_1+\sum_{m=0}^{k-1}\log|q(ze((m+1)\omega))q(ze((m)\omega))|
\leq \log \|T_{[1,N]}(z)\|-\log \|T_{[1,N]}(ze(k\omega))\|
\nonumber\\&\ &\ \ \ \ \ \ \ \ \ \ \ \ \ \ \ \ \ \ \leq
kC_1-\sum_{m=0}^{k-1}\log|q(ze((N+m-1)\omega))q(ze((N+m)\omega))|\nonumber\end{eqnarray}
\item[{\rm{(2)}}]\begin{eqnarray}&&-KC_1+\sum_{k=0}^{K-1}\frac{K-k}{K}\log|q(ze(k\omega))q(ze((k+1)\omega))|
\leq \log \|T_{[1,N]}(z)\|-\frac{1}{K}\sum_{k=1}^{K}\log
\|T_{[1,N]}(ze(k\omega))\|\nonumber\\&\ &\ \ \ \ \ \ \ \ \ \ \ \ \ \
\ \  \ \leq
KC_1-\sum_{k=0}^{K-1}\frac{K-k}{K}\log|q(ze((k-1+N)\omega))q(ze((k+N)\omega))|.\nonumber\end{eqnarray}\end{enumerate}\end{corollary}
\begin{proof}
(1) is obvious;\\

(2)For $1\leq k\leq
K$,\begin{eqnarray}&&-\frac{kC_1}{K}+\sum_{m=0}^{k-1}\frac{1}{K}\log|q(ze((m+1)\omega))q(ze((m)\omega))|
\leq \frac{1}{K}\log \|T_{[1,N]}(z)\|-\frac{1}{K}\log
\|T_{[1,N]}(ze(k\omega))\| \nonumber\\&\ &\ \ \ \ \ \ \ \ \ \ \ \ \
\ \ \ \ \ \leq
\frac{kC_1}{K}-\sum_{m=0}^{k-1}\frac{1}{K}\log|q(ze((N+m-1)\omega))q(ze((N+m)\omega))|\nonumber.\end{eqnarray}
As
\[\sum_{k=1}^K\sum_{m=0}^{k-1}\frac{1}{K}\log|q(ze((m+1)\omega))q(ze((m)\omega))|=\sum_{k=0}^{K-1}\frac{K-k}{K}\log|q(ze((k+1)\omega))q(ze((k)\omega))|
\]
and \[\sum_{k=1}^K\frac{kC_1}{K}\leq \sum_{k=1}^K C_1=KC_1,\] then
\[\log \|T_{[1,N]}(z)\|-\sum_{k=1}^K\frac{1}{K}\log
\|T_{[1,N]}(ze(k\omega))\|\geq
-KC_1+\sum_{k=0}^{K-1}\frac{K-k}{K}\log|q(ze((k+1)\omega))q(ze((k)\omega))|.
\]
Also\[ \log \|T_{[1,N]}(z)\|-\sum_{k=1}^K\frac{1}{K}\log
\|T_{[1,N]}(ze(k\omega))\|\leq
KC_1-\sum_{k=0}^{K-1}\frac{K-k}{K}\log|q(ze((k-1+N)\omega))q(ze((k+N)\omega))|.\]
\end{proof}

\begin{lemma}
\label{lem:riesz} Let $u:\Omega\to \IR$ be a subharmonic function on
a domain $\Omega\subset\IC$. Suppose that $\partial \Omega$ consists
of finitely many piece-wise $C^1$ curves. There exists a positive
measure $\mu$ on~$\Omega$ such that for any $\Omega_1\Subset \Omega$
(i.e., $\Omega_1$ is a compactly contained subregion of~$\Omega$)
\begin{equation}
\label{eq:rieszrep} u(z) = \int_{\Omega_1}
\log|z-\zeta|\,d\mu(\zeta) + h(z)
\end{equation}
where $h$ is harmonic on~$\Omega_1$ and $\mu$ is unique with this
property. Moreover, $\mu$ and $h$ satisfy the bounds \be
\mu(\Omega_1) &\le& C(\Omega,\Omega_1)\,(\sup_{\Omega} u - \sup_{\Omega_1} u) \label{21002} \\
\|h-\sup_{\Omega_1}u\|_{L^\infty(\Omega_2)} &\le&
C(\Omega,\Omega_1,\Omega_2)\,(\sup_{\Omega} u - \sup_{\Omega_1} u)
\label{21003} \ee for any $\Omega_2\Subset\Omega_1$.
\end{lemma}

For the proof See Lemma 2.2 in [GS1].

\begin{theorem}\label{22002} Let $u$ be a subharmonic function
defined in the annulus $A_\rho$. Suppose furthermore that
$|u(z)|\leq 1$. Then for any $1-\frac{\rho}{2}\leq r\leq
1+\frac{\rho}{2}$\be
\mes(\{x:|\sum_{k=1}^nu(re(x-k\omega))-n<u(re(\cdot))>|>\delta
n\})<\exp(-c\delta n+s_n),\ee where
$<u(re(\cdot))>:=\int_0^1u(re(y))dy$, $s_n\leq C(\log n)^A$ for
general n and $s_n \leq C\log n$ if $n=q_j$ for any j.\end{theorem}

For the proof see Theorem
3.8 in [GS].

\begin{remark}\begin{enumerate}
\item[{\rm{(1)}}]The constants $c,C$ here do not depend on $\delta$.

\item[{\rm{(2)}}]Actually, the condition $|u(z)|\leq 1$ can be replaced by
$u(z)=\int\log|z-\zeta|d\mu(\zeta)+h(z)$,with $\|\mu\|+\|h\|\leq
C$, see the proof of Theorem
3.8 in [GS]
\end{enumerate}
\end{remark}

In what follows we will use the  following version

\begin{theorem}\label{22004} Let $u$ be a subharmonic function
defined in the annulus $A_\rho$. Suppose furthermore that
$u(z)=\int\log|z-\zeta|d\mu(\zeta)+h(z)$, with
$\mu(A_{\frac{\rho}{2}})+\|h\|_{L^{\infty}(A_{\frac{\rho}{2}})}\leq
\hat{C}$. Then for any $1-\frac{\rho}{2}\leq r\leq
1+\frac{\rho}{2}$\be
\mes(\{x:|\sum_{k=1}^nu(re(x+k\omega))-n<u(re(\cdot))>|>\delta
n\})<\exp(-c\delta n).\ee where $c=c(\hat{C},\omega)$.\end{theorem}

Set $u_N(z,E,\omega)=\frac{1}{N}\log \|T_N(z,E,\omega)\|$,
$\tilde{u}_N(z,E,\omega)=\frac{1}{N}\log
\|\tilde{M}_N(z,E,\omega)\|$. Sometimes we use $u_N(z)$ or $u_N$ for
short. It is also the same for $\tilde{u}_N(z,E,\omega)$. Let
$L_{N,r}(E)=<\tilde{u}_N(re(\cdot))>,\ D_r=<\log|q(re(\cdot))|>,\
J_{N,r}(E)=<u_N(re(\cdot))>=L_{N,r}(E)+D_r$. For $r=1$ we use the
notations $L_N(E),\ D,\ J_N(E)$.
\begin{lemma}\label{2200a} Function $u_N(z)$ is subharmonic in
$A_{\rho}$ and obeys $u_N(z)\leq C(p,q)$.
\end{lemma}
\begin{proof}Since $B(z)$ is analytic in $A_\rho$, so is $T_N(z)$.
Therefore $u_N(z)=\frac{1}{N}\log\|T_N(z)\|$ is subharmonic in
$A_\rho$. The estimate $u_N(z)\leq C(p,q)$ follows from Lemma
\ref{lem12002}.\end{proof}

\begin{lemma}\label{2200b}There exists $C_2(q)<+\infty$ s.t.
\[\sup_{x\in \mathbb{T}}u_N(e(x))\geq -C_{2}(q).\]
\end{lemma}
\begin{proof} One has $\|A\|^2\geq |\det A|$ for any $2\times 2$
matrix. So \[u_N(z)\geq \frac{1}{2N}\log |\det
T_N(z)|=\frac{1}{2N}\sum_{n=0}^{N-1}\log
|q(ze(n\omega))q(ze((n+1)\omega))|,\] see (\ref{14001}). Recall that
$\log|q(e(x))|$ is integrable. So, $\frac{1}{2}\log
|q(e(x))q(e(x+\omega))|$ is integrable. So,
\be\label{14003}\int_{\mathbb{T}}\left |\frac{1}{2}\log
|q(e(x))q(e(x+\omega))|\right |dx=C_{2}(q)<+\infty.\ee Therefore
\[\int_{\mathbb{T}}\left |\frac{1}{2N}\sum_{n=0}^{N-1}\log
|q(e(x+n\omega))q(e(x+(n+1)\omega))|\right |dx\leq
C_{2}(q)<+\infty.\] Hence,
\[\sup_{x\in\mathbb{T}}\frac{1}{2N}\sum_{n=0}^{N-1}\log
|q(e(x+n\omega))q(e(x+(n+1)\omega))|\geq -C_{2}(q).\]
\end{proof}
\begin{lemma}\label{2200c}
One has
\[u_N(z)=\int \log|z-\zeta|d\mu(\zeta)+h(z),\]
where
\[\mu(A_{\frac{\rho}{2}})\leq C_3(p,q),\ \|h\|_{L^{\infty}(A_{\frac{\rho}{2}})}\leq
C_3(p,q).\]
\end{lemma}
\begin{proof}
The statement follows from Lemma \ref{2200a}, \ref{2200b}, and
(\ref{21002}), (\ref{21003}).
\end{proof}
\begin{remark}\label{25002}By
(\ref{14003}), we have
\[\sup_{x\in\mathbb{T}}\frac{1}{2}\log
|q(re(x))q(re(x+\omega))|\geq -C_{2}(q).\] Then by Lemma
\ref{2200b}, and (\ref{21002}), (\ref{21003}), one has
\[\frac{1}{2}\log
|q(re(x))q(re(x+\omega))|=\int
\log|z-\zeta|d\mu_q(\zeta)+h_q(\zeta),\] where
\[\mu_q(A_{\frac{\rho}{2}})\leq C_q,\ \|h\|_{L^{\infty}(A_{\frac{\rho}{2}})}\leq
C_q.\]
\end{remark}
~\\
~\\
 By Theorem \ref{22004} and Lemma \ref{2200c},
we have
\begin{lemma}\label{22005}For any $1-\frac{\rho}{2}\leq r\leq
1+\frac{\rho}{2}$, $\delta$ and K, \be
\mes(\{x:|\sum_{k=1}^Ku_N(re(x+k\omega))-K<u_N(re(\cdot))>|>\delta
K\})<\exp(-c\delta K),\ee where $c=c(p,q,\omega)$.
\end{lemma}
\begin{remark}\label{23006} \begin{enumerate}
\item[{\rm{(1)}}]By Remark \ref{25002} and Theorem
\ref{22004}, also for any $1-\frac{\rho}{2}<r<1+\frac{\rho}{2}$,
$\delta$ and K
\[\mes(\{x:|\sum_{k=1}^K\frac{1}{2}\log
|q(re(x+k\omega))q(re(x+(k+1)\omega))|-K<\frac{1}{2}\log
|q(re(x))q(re(x+\omega))|>|>\delta K\})<\exp(-c_q\delta K).
\]

\item[{\rm{(2)}}]It is well-known that  $\int_0^1 \left |
\log |q(re(x))|\right |dx \leq C'_q$, for any $1-\frac{\rho}{2}\leq
r \leq 1+\frac{\rho}{2}$, if $b(x)$ is analytic.

\item[{\rm{(3)}}]Due to (\ref{10004}) one has
\[\tilde{u}_N(re(x))=u_N(re(x))-\frac{1}{2N}\sum_{k=1}^N\log
|q(re(x+k\omega))q(re(x+(k+1)\omega))|.\]Hence,
\[\mes\{x:|\sum_{k=1}^K\tilde{u}_N(re(x+k\omega))-K<\tilde{u}_N(re(\cdot))>|>\delta
K\}\leq\exp(-c\delta K)+K\exp(-c_q\delta K).\]What's more, there
exists $\check{K}=\check{K}(p,q,\delta)$, s.t. for any $K\geq
\check{K}$,
\[\mes\{x:|\sum_{k=1}^K\tilde{u}_N(re(x+k\omega))-K<\tilde{u}_N(re(\cdot))>|>\delta
K\}\leq \exp(-\hat{c}\delta K).\]
\end{enumerate}
\end{remark}
~\\
~\\
\begin{lemma}\label{23007}For any $1-\frac{\rho}{2}\leq r\leq
1+\frac{\rho}{2}$, $\delta$ and  K
\[\mes(\{x:|\sum_{k=0}^{K-1}\frac{K-k}{K}\log|q(re(x+k\omega))q(re((x+k+1)\omega))|-(K+1)D_r|>\delta
K\})\leq K\exp(-c_q\delta K),\] where $c_q$ is as in Remark
\ref{23006}.

\end{lemma}
\begin{proof}
Set
\[\mathbb{X}_r:=\{x:|\sum_{k=0}^{K-1}\frac{K-k}{K}\log|q(re(x+k\omega))q(re((x+k+1)\omega))|-(K+1)D_r|>\delta
K\},\]and\[Q_r(x)=\log
|q(re(x))q(re(x+\omega))|-\int_{\mathbb{T}}\log
|q(re(x))q(re(x+\omega))|dx=\log |q(re(x))q(re(x+\omega))|-2D_r.\]
Note also that $\sum_{k=0}^{K-1}\frac{K-k}{K}=\frac{K+1}{2}$, and
\[\sum_{k=0}^{K-1}\frac{K-k}{K}\log|q(re(x+k\omega))q(re((x+k+1)\omega))|-(K+1)D_r=\sum_{k=0}^{K-1}\frac{K-k}{K}Q_r(x+k\omega).\]
So\[\left
|\sum_{k=0}^{K-1}\frac{K-k}{K}\log|q(re(x+k\omega))q(re((x+k+1)\omega))|-(K+1)D_r\right
|>\delta K\]
\[\Leftrightarrow |\sum_{k=0}^{K-1}\frac{K-k}{K}Q_r(x+k\omega)|>\delta K,\ \ \mbox{and } <Q_r>=0.\]
Define $S_{k,r}(x)=\sum_{j=0}^{k-1}Q_r(x+j\omega)$, then
\[\sum_{k=0}^{K-1}\frac{K-k}{K}Q_r(x+k\omega)=\frac{1}{K}\sum_{k=1}^{K}S_k(x).\]
Set $\mathbb{X}_{k,r}:=\{x:|S_{k,r}(x)|>\delta K\}$, then
\[\mathbb{X}_r\subseteq \bigcup_{k=1}^K \mathbb{X}_{k,r},\ \ \mes(\mathbb{X}_r)\leq \sum_{k=1}^{K}\mes(\mathbb{X}_{k,r}).\]
Note that
\[\mathbb{X}_{k,r}=\{x:|\sum_{j=0}^{k-1}\log |q(re(x+j\omega))q(re(x+(j+1)\omega))|-k\int_{\mathbb{T}}\log
|q(re(x))q(re(x+\omega))|dx|>\delta K\}.\]  Thus,by Remark
\ref{23006}
\[\mes \mathbb{X}_{k,r}=\{x:|S_{k,r}(x)|>\delta K= \frac{\delta K}{k}k \})\leq
\exp(-c_q\frac{\delta K}{k} k)=\exp(-c_q \delta K).\] Thus
\[\mes \mathbb{X}_r\leq K\times \exp(-c_q \delta K).\]
\end{proof}
\begin{theorem}\label{22006}There exists $\check{N}(p,q,\omega)$ such that for any $N\geq \check{N}$, any $1-\frac{\rho}{2}\leq r\leq
1+\frac{\rho}{2}$ and $\delta<1$
\[\mes(\{x:|\frac{1}{N}\log\|T_N(re(x))\|-J_{N,r}(E) |>\delta
\})<\exp(-\check{c}\delta^2 N),\]where
$\check{c}=\check{c}(p,q,\omega)$.\end{theorem}
 \begin{proof}
Set
\[\mathbb{Y}_r:=\{x:|\frac{1}{N}\log\|T_N(re(x))\|-\frac{1}{N}\int\log\|T_N(re(x))\|dx |>\delta
\}.\]Then \begin{eqnarray}\mathbb{Y}_r&\subseteq &
\{x:|\frac{1}{K}\sum_{k=1}^K
u_N(re(x+k\omega))-<u_N(re(\cdot))>|>\frac{\delta}{2}
\}\bigcup\{x:u_N(re(x))-\frac{1}{K}\sum_{k=1}^K
u_N(re(x+k\omega))>\frac{\delta}{2}\}\nonumber\\
&\ &\ \ \ \ \ \ \bigcup\{x:u_N(re(x))-\frac{1}{K}\sum_{k=1}^K
u_N(re(x+k\omega))<-\frac{\delta}{2}\}\nonumber\\
&:=&\mathbb{Y}_{0,r}\bigcup
\mathbb{Y}_{+,r}\bigcup\mathbb{Y}_{-,r}.\nonumber
\end{eqnarray}
Set $K=C_4\delta N$. By Lemma \ref{22005},\[\mes
\mathbb{Y}_{0,r}<\exp(-\frac{c}{2}\delta K)=\exp(-C_4\times
c\delta^2 N ).\] We need to estimate $\mes \mathbb{Y}_{\pm,r}$. Due
to part (2) of Corollary \ref{12004}, one has
\[\mathbb{Y}_{+,r}\subseteq
\{x:KC_1-\sum_{k=0}^{K-1}\frac{K-k}{K}\log|q(ze((k-1+N)\omega))q(ze((k+N)\omega))|>\frac{N\delta}{2}\},\]
\[\mathbb{Y}_{-,r}\subseteq
\{x:-KC_1+\sum_{k=0}^{K-1}\frac{K-k}{K}\log|q(ze(k\omega))q(ze((k+1)\omega))|
<-\frac{N\delta}{2}\}.\] Set $C_4<\frac{1}{4C_1}$, then
\[KC_1-\sum_{k=0}^{K-1}\frac{K-k}{K}\log|q(ze((k-1+N)\omega))q(ze((k+N)\omega))|>\frac{N\delta}{2}\]
\[\Rightarrow
\sum_{k=0}^{K-1}\frac{K-k}{K}\log|q(ze((k-1+N)\omega))q(ze((k+N)\omega))|<-\frac{\delta
N}{4}=-\frac{K}{4C_4}.\]
\begin{eqnarray}\label{25050}&\ &\sum_{k=0}^{K-1}\frac{K-k}{K}\log|q(ze((k-1+N)\omega))q(ze((k+N)\omega))|
-\sum_{m=0}^{K-1}2\frac{K-m}{K}\int_{\mathbb{T}}\log|q(re(x))|dx\\
&=&\sum_{k=0}^{K-1}\frac{K-k}{K}\log|q(ze((k-1+N)\omega))q(ze((k+N)\omega))|-(K+1)D_r\nonumber\\
&<&-\frac{K}{4C_4}-(K+1)D_r.\nonumber\end{eqnarray} Recall that
\[ \tilde{D}=\max_{1-\frac{\rho}{2}\leq r\leq
1+\frac{\rho}{2}}|D_r|<+\infty.\]Let $C_4<\frac{1}{8|\tilde{D}|}$ to
make $\frac{1}{4C_4}+D_r=C_{5,r}>\tilde{D}>0$. Note that $C_{5,r}$
is also continue for $r$. Thus
\[C_5=\min_{1-\frac{\rho}{2}\leq r\leq
1+\frac{\rho}{2}}C_{5,r}>0\] Then there exists $C'_5=C'_5(p,q)>0$,
s.t.
\[\frac{K}{4C_4}+(K+1)D_r=C_{5,r}K+D_r\geq C_5K+D_r\geq C'_5 K \] for any K. Thus Lemma \ref{23007} applies for
$\delta=C'_5$,
\[\mes(\mathbb{Y}_{+,r})\leq K\exp(-c_q \times C'_5
K)=C_4\delta N\exp(-c_q \times C'_5 \times C_4\delta N).\]  As
$y\exp(-\xi y)\leq \xi^{-1}$ for any $y, \xi>0$,
\[\mes(\mathbb{Y}_{+,r})\leq C_4\delta N\exp(-\frac{c_q \times C'_5}{2} \times C_4\delta N)\times\exp(-\frac{c_q \times C'_5}{2} \times C_4\delta N)
<\frac{2}{c_q\times C'_5}\exp(-\frac{c_q \times C'_5}{2}\times
C_4\delta N)<\exp(-c_{p,q}\delta N),\]for $N\geq \check{N}$, where
$\check{N}$ depends on $c_{p,q}$ and $C'_5$, i.e.
$\check{N}=\check{N}(p,q,\omega)$.
Similarly,\[\mes(\mathbb{Y}_{-,r})<\exp(-c_{p,q}\delta N).\] So for
$N\geq\check{N}(p,q,\omega)$
\[\mes\mathbb{Y}_r<2\exp(-c_{p,q}\delta N)+\exp(-C_4\times c\delta^2 N )<\exp(-\check{c}\delta^2
N),\] where $\check{c}=\check{c}(p,q,\omega)$.
\end{proof}
\begin{remark}\begin{enumerate}
\item[{\rm{(1)}}]
Recall that
\[\tilde{u}(z)=u_N(z)-\frac{1}{2N}\sum_{n=0}^{N-1}\log
|q(re(x+n\omega))q(re(x+(n+1)\omega))|,\
<\tilde{u}_N(z)=<u_N(z)>-D_r,\] see (\ref{10004}). By Remark
\ref{23006}, Theorem \ref{22006} , for any $N\geq \check{N}$
\[\mes(\{x:|\tilde{u}_N(re(x))-L_{N,r}(E)|>2\delta\})<\exp(-c_q\delta K)+\exp(-\check{c}\delta^2 N)<\exp(-4\tilde{c}\delta^2
N),\]where $L_{N,r}(E)=J_{N,r}-D_r$. Hence for any $N\geq \check{N}$
\begin{equation}\label{25001}\mes(\{x:|\tilde{u}_N(re(x))-L_{N,r}(E)|>\delta\})<\exp(-\tilde{c}\delta^2
N).\end{equation}

\item[{\rm{(2)}}]Once again let us note that the constants $c,\
\tilde{c}$ here do not depend on $\delta$. In particular, one can
choose here $\delta$ depending on N.
\end{enumerate}
\end{remark}
\begin{lemma}
\label{lem:liprad} Let $1>\rho>0$ and suppose $u$ is subharmonic
on~$A_\rho$ such that $\sup_{z\in A_\rho} u(z)\le 1$ and
$\int_{\tor} u(re(x))\,dx\ge0$. Then for any $r_1,r_2$ so that
$1-\frac{\rho}{2}\leq r_1,r_2\leq 1+\frac{\rho}{2}$ one has
\[ |\langle u(r_1 e(\cdot)) \rangle - \langle u(r_2e(\cdot)) \rangle| \le C_\rho\,|r_1-r_2|.\]
\end{lemma}
\begin{proof}
See [GS1] Lemma 4.1.
\end{proof}
\begin{remark}\label{25018}It is easy to see that
this lemma also holds for $u_N(z)$ with $C_\rho=C_\rho(p,q,\omega)$.
Thus\[|\int_0^1 u_N(re(\theta))d\theta-J_{N,1}(E) d\theta|\leq
C_{\rho}|r-1|\] for any $1-\frac{\rho}{2}<r<1+\frac{\rho}{2}$.
\end{remark}
~\\
~\\
\begin{lemma}\label{22010}For any $N\geq \check{N}(p,q,\omega)$\be\frac{1}{N}\log\|T_N(e(x),E)\|\leq J_{N,1} +C_6(\frac{\log
N}{N})^{\frac{1}{2}},\ee where $C_6=C_6(p,q,\omega)$ and $\check{N}
$ is as in Theorem \ref{22006}.\end{lemma}\begin{proof} Let
$0<\delta<\frac{\rho}{4}$ be arbitrary. Note that $e(x+iy)=e^{-2\pi
y}e(x)$, $1-\frac{\rho}{4}<1-\frac{\delta}{C'_\rho}\leq e^{-2\pi
y}\leq 1+\frac{\delta}{C'_\rho}<1+\frac{\rho}{4}$, if $|y|\leq
\frac{\delta}{4\pi e C'_\rho}$, where $C'_\rho=\max(1,C_\rho)$ and
$e=\exp(1)$. By Remark \ref{25018} one has \be\label{25019}
|\int_0^1 u_N(re(\theta))d\theta-J_{N,1}(E) d\theta|\leq \delta,\
\mbox{if}\ |y|\leq \frac{\delta}{4\pi e C'_\rho}.\ee Set
\[\mathbb{B}_y:=\{x:|u_N(e(x+iy))-J_{N,1}|>2\delta\}.\]
It follows from (\ref{25019}) that for $|y|\leq \frac{\delta}{4\pi e
C'_\rho}$ holds
\[\mathbb{B}_y\subseteq
\{x:|u_N(e(x+iy))-\int_0^1u_N(e(\theta+iy))d\theta|>\delta\}.\] Due
to Theorem \ref{22006} one obtains $\mes \mathbb{B}_y\leq
\exp(-\check{c}\delta^2 N)$. The function $u_N(e(x+iy))$ is
subharmonic, for $e(x+iy)\in A_\rho$. let $x_0$ be arbitrary and
$y_0=0$. Then $e(x_0)\in A_{\frac{\rho}{4}}$. Due to subharmonicity
one has for any $t_0<\frac{\rho}{4}$
\begin{eqnarray}u_N(e(x_0))-J_{N,1}&\leq& \frac{1}{\pi
t_0^2}\iint_{|(x,y)-(x_0,0)|\leq
t_0}[u_N(e(x+iy))-J_{N,1}]dxdy\nonumber\\
&=&\frac{1}{\pi t_0^2}\int_{|y|\leq t_0}\int_{|x- x_0|\leq
\sqrt{t_0^2-|y|^2}}[u_N(e(x+iy))-J_{N,1}]dxdy.\nonumber\end{eqnarray}
Furthermore
\[\int_{|x- x_0|\leq
\sqrt{t_0^2-|y|^2}}[u_N(e(x+iy))-J_{N,1}]dx=\left (\int_{\{|x-
x_0|\leq \sqrt{t_0^2-|y|^2}\}\bigcap \mathbb{B}_y}+\int_{\{|x-
x_0|\leq \sqrt{t_0^2-|y|^2}\}\setminus \mathbb{B}_y}\right
)[u_N(e(x+iy))-J_{N,1}]dx.\]Note that
\[|u_N(e(x+iy))-J_{N,1}|\leq 2\delta,\ \mbox{if } x\not\in
\mathbb{B}_y \mbox{ and } y<\frac{\delta}{4e\pi C'_\rho}.\]So
\[\left |\int_{\{|x-
x_0|\leq \sqrt{t_0^2-|y|^2}\}\setminus
\mathbb{B}_y}[u_N(e(x+iy))-J_{N,1}]dx\right |\leq
2\delta\times(2\sqrt{t_0^2-|y|^2}).\]Due to Cauchy-Schwartz
inequality
\[\left |\int_{\{|x-
x_0|\leq \sqrt{t_0^2-|y|^2}\}\bigcap
\mathbb{B}_y}[u_N(e(x+iy))-J_{N,1}]dx\right |\leq \left
(\int_0^1|u_N(e(x+iy))-J_{N,1}|^2dx\right )^{\frac{1}{2}}\left (\mes
\mathbb{B}_y\right )^{\frac{1}{2}}\leq
C_7\exp(-\frac{\check{c}}{2}\delta^2 N).\] Set
$t_0=\frac{\delta}{4e\pi C'_\rho}$, then
\begin{eqnarray}u_N(e(x))-J_{N,1}&\leq &\frac{1}{\pi t_0^2}\int_{|y|\leq
t_0}[C_7\exp(-\frac{\check{c}}{2}\delta^2
N)+2\delta\times(2\sqrt{t_0^2-|y|^2})]dy\nonumber\\
&\leq &\frac{1}{\pi t_0^2}\times
C_7\exp(-\frac{\check{c}}{2}\delta^2
N)\times (2t_0)+2\delta\nonumber\\
&=&\frac{8e C_7C'_\rho}{\delta}\exp(-\frac{\check{c}}{2}\delta^2
N)+2\delta\nonumber.\end{eqnarray}
 Set $\delta=(\frac{C_8\log N}{N})^{\frac{1}{2}}$, where
 $C_8>\frac{2}{\check{c}}$. Then
$\exp(-\frac{\check{c}}{2}C_8\log N)<\frac{1}{N}$, and
\[u_N(e(x))\leq J_{N,1} +8e C_7C'_\rho\times(\frac{N}{C_8\log N})^{\frac{1}{2}}\times \frac{1}{N}+2(\frac{C_8\log N}{N})^{\frac{1}{2}}\leq J_{N,1} +C_6(\frac{\log
N}{N})^{\frac{1}{2}}.\]
\end{proof}
\begin{lemma}\label{22020}For any $0\leq x\leq 1$ and any $N\geq \check{N}$ holds
\[\log\|\tilde{M}_N(e(x),E)\|\leq NL_N+C_6(N\log
N)^{\frac{1}{2}}-NF_N(x),\] where
\[F_N(x)=\frac{1}{2N}\sum_{n=0}^{N-1}Q(x+n\omega),\ \ \ \ \
Q(x)=\log|q(e(x))q(e(x+\omega))|-2D.\]\end{lemma}
\begin{proof}
\begin{eqnarray}\label{20019}
\log\|\tilde{M}_{[1,N]}(e(x),E)\|&=&\log
\|T_{[1,N]}(e(x),E)\|-\frac{1}{2}\sum_{n=0}^{N-1}\log
|q(e(x+n\omega))q(e(x+(n+1)\omega))|
\\
&\leq& NJ_N(E) +C_6(N\log
N)^{\frac{1}{2}}-\frac{1}{2}\sum_{n=0}^{N-1}\log
|q(e(x+n\omega))q(e(x+(n+1)\omega))|\nonumber.
\end{eqnarray}
Recall that
\[J_N(E)=L_N(E)+D.\]
Then due to (\ref{20019}) one has
\[\log\|\tilde{M}_N(e(x),E)\|\leq NL_N+C_6(N\log
N)^{\frac{1}{2}}-NF_N(x).\]
\end{proof}
\begin{remark}\label{20101}Note that Lemma \ref{22020} implies, in
particular that $NL_N+C_6(N\log N)^{\frac{1}{2}}-NF_N(x)\geq 0$ for
any $x$ for large N.\end{remark}

\begin{lemma}\label{22021}For any $0\leq x\leq 1$ and any $k\geq \check{N}$ holds
\[\left
|\log\|\tilde{M}_N(e(x+k\omega),E)\|-\log\|\tilde{M}_N(e(x))\|\right|\leq
2kL_k(E)+2C_6(N\log N)^{\frac{1}{2}}-kF_k(x)-kF_k(x+N\omega).\]
\end{lemma}
\begin{proof}
One has
\[\tilde{M}_N(e(x+k\omega),E)\tilde{M}_k(e(x),E)=\tilde{M}_k(e(x+N\omega),E)\tilde{M}_N(e(x),E).\]
Then
\[\left
|\log\|\tilde{M}_N(e(x+k\omega),E)\|-\log\|\tilde{M}_N(e(x))\|\right|\leq
\log\|\tilde{M}_k(e(x),E)\|+\log\|\tilde{M}_k(x(x+N\omega),E)\|,\]because
$\|A^{-1}\|=\|A\|\geq 1$ if $\det A=1$. Due to Lemma \ref{22020},
\[\log\|\tilde{M}_k(e(x),E)\|+\log\|\tilde{M}_k(e(x+N\omega),E)\|\leq
2kL_k(E)+2C_6(N\log N)^{\frac{1}{2}}-kF_k(x)-kF_k(x+N\omega).\]
\end{proof}
\begin{remark}\label{23023}Due to Lemma \ref{lem12002}
\[u_N(e(x),E)\leq \log C(p,q)\] for any $x\in \mathbb{T}$, any $N$ and any
$E$. Similarly,
\[\left
|\log\|\tilde{M}_N(e(x+k\omega),E)\|-\log\|\tilde{M}_N(e(x))\|\right|\leq
2k(\log C(p,q)-D)-kF_k(x)-kF_k(x+N\omega)\] for any $x\in
\mathbb{T}$, any $N$, any $k$ and any $E$.

\end{remark}

\section{using the avalanche principle}

\begin{prop}
\label{prop:AP} Let $A_1,\ldots,A_n$ be a sequence of  $2\times
2$--matrices whose determinants satisfy
\begin{equation}
\label{eq:detsmall} \max\limits_{1\le j\le n}|\det A_j|\le 1.
\end{equation}
Suppose that \be
&&\min_{1\le j\le n}\|A_j\|\ge\mu>n\mbox{\ \ \ and}\label{large}\\
   &&\max_{1\le j<n}[\log\|A_{j+1}\|+\log\|A_j\|-\log\|A_{j+1}A_{j}\|]<\frac12\log\mu\label{diff}.
\ee Then
\begin{equation}
\Bigl|\log\|A_n\cdot\ldots\cdot A_1\|+\sum_{j=2}^{n-1}
\log\|A_j\|-\sum_{j=1}^{n-1}\log\|A_{j+1}A_{j}\|\Bigr| <
C\frac{n}{\mu} \label{eq:AP}
\end{equation}
with some absolute constant $C$.
\end{prop}
\begin{proof}See [GS].\end{proof}
\begin{remark}For the rest of the paper, we do not use $e(x+iy)$
with $y\not=0$. For that reason we write $x$ instead of $e(x)$ in
all expressions. What's more, without special statement, $N\geq
\check{N}$ and $N\geq \check{K}$ from now on($\delta$ in $\check{K}$
will be defined in Lemma \ref{22012}).\end{remark}

\begin{lemma}\label{62001}Let $\tilde{c}$ be as in (\ref{25001}).
Let $L_N(E)>100\delta>0$, where $\delta<1$ is a constant not
depending on $N$, and $L_{2N}(E)>\frac{9}{10}L_{N}(E)$.  Let $N'=m
N$, $m\in \mathbb{N}$ and $ m\leq \exp(\frac{\tilde{c}}{4}\delta^2
N) $. Then \[|L_{N'}(E)+L_{N}(E)-2L_{2N}(E)|\leq
\exp(-\tilde{c}'\delta^2 N)+\frac{2}{9m}L_N(E),\]where
$\tilde{c}'=\tilde{c}'(p,q, \omega)$. If
$\exp(\frac{\tilde{c}}{10}\delta^2 N) \leq m\leq
\exp(\frac{\tilde{c}}{4}\delta^2 N) $, we have \be
|L_{N'}(E)+L_{N}(E)-2L_{2N}(E)|\leq \exp(-\hat{c}\delta^2 N),\ee
where $\hat{c}=\hat{c}(p,q, \omega)$.
\end{lemma}
\begin{proof}
By (\ref{25001}), we have for $0\leq j\leq m-1$
\[|\tilde{u}_{N}(x+jN\omega,E)-L_{N}(E)|<\delta\]
\[|\tilde{u}_{2N}(x+jN\omega,E)-L_{2N}(E)|<\delta\]
for $x\in\mathbb{G}_1$, with
\[\mes(\mathbb{T}\backslash \mathbb{G}_1)\leq 2m \times\exp(-\tilde{c}\delta^2 N)<\exp(-\frac{2\tilde{c}}{3}\delta^2 N).\]
Thus when $x\in\mathbb{G}_1$,
\[\|\tilde{M}_{N}(x+jN\omega,E)\|>\exp(N(L_{N}(E)-\delta))>\exp(\frac{99}{100}NL_{N}(E)),\]
and
\begin{eqnarray}
&&\left |\log\|\tilde{M}_{N}(x+jN\omega,E)\|+\log\|\tilde{M}_{N}(x+(j+1)N\omega,E)\|-\log\|\tilde{M}_{N}(x+jN\omega,E)\tilde{M}_{N}(x+(j+1)N\omega,E)\|\right |\\
&\ &\ \ \ \ \ \ \ \ <4N\delta+2N|L_{N}(E)-L_{2N}(E)|
<\frac{6}{25}NL_{N}(E),\nonumber\end{eqnarray}Since $0\leq
L_N(E)-L_{2N}(E)<\frac{1}{10}L_N(E).$  One has \[
\tilde{M}_{N'}(x,E)=\prod_{j=m}^1
\tilde{M}_N(x+(j-1)N\omega,E).\]The avalanche principle applies for
$\mu=\exp(\frac{1}{2}NL_N(E))$. Integrating over $\mathbb{G}_1$ one
obtains
\begin{equation}\label{60201}
|\int_{\mathbb{G}_1}\tilde{u}_{N'}(x,E)dx+\int_{\mathbb{G}_1}\sum_{j=2}^{m-1}\tilde{u}_{N}(x+(j-1)N\omega,E)dx-\int_{\mathbb{G}_1}\sum_{j=1}^{m-1}
\tilde{u}_{2N}(x+(j-1)N\omega,E)dx|\leq
C\frac{m}{N'}\exp(-\frac{1}{2}NL_N(E)),\end{equation}where
$N'=m\times N$.  We want to replace here the integration over
$\mathbb{G}_1$ by integration over $\mathbb{T}$. Recall that due to
(4) in Remark \ref{13002}
\[\int_{\mathbb{T}}\tilde{u}_n^2(E) dx\leq \tilde{C}(p,q)\] for any $n$ and any $E$. Hence,
by Cauchy-Schwartz inequality
\be\label{11111}|\int_{\mathbb{B}}\tilde{u}_K(E)dx|\leq
\tilde{C}(p,q)^{\frac{1}{2}}(\mes \mathbb{B})^{\frac{1}{2}}\ee for
any $K$, any $E$ and any $\mathbb{B}\subseteq \mathbb{T}$. Hence
\[|\int_{\mathbb{T}\backslash\mathbb{G}_1}\tilde{u}_K(E)dx|\leq
\tilde{C}(p,q)^{\frac{1}{2}}\exp(-\frac{\tilde{c}}{3}\delta^2N)\]
for any $K$ and any $E$. Thus
\begin{equation}\label{60202}|\int_{\mathbb{T}\backslash\mathbb{G}_1}\tilde{u}_{N'}(x,E)dx+\frac{1}{m}\int_{\mathbb{T}\backslash\mathbb{G}_1}\sum_{j=2}^{m-1}\tilde{u}_{N}(x+(j-1)N\omega,E)dx-\frac{2}{m}\int_{\mathbb{T}\backslash\mathbb{G}_1}\sum_{j=1}^{m-1}
\tilde{u}_{2N}(x+(j-1)N\omega,E)dx|\leq
4\tilde{C}(p,q)^{\frac{1}{2}}\exp(-\frac{\tilde{c}}{3}\delta^2N).
\end{equation}
Combining (\ref{60201}) with (\ref{60202}), one has
\begin{equation}
|L_{N'}(E)+\frac{m-2}{m}L_N(E)-\frac{2(m-1)}{m}L_{2N}(E)| \leq
4\tilde{C}(p,q)^{\frac{1}{2}}\exp(-\frac{\tilde{c}}{3}\delta^2N)+C\frac{m}{N'}\exp(-\frac{1}{2}NL_N(E))\leq
\exp(-\tilde{c}'\delta^2N)\nonumber.\end{equation}Thus
\begin{eqnarray} |L_{N'}(E)+L_N(E)-2L_{2N}(E)|&\leq &
\exp(-\tilde{c}'\delta^2N)+\frac{2}{m}|L_N(E)-L_{2N}(E)|\\
&<& \exp(-\tilde{c}'\delta^2N)+\frac{2}{m}\times
\frac{1}{10}L_N(E)\nonumber\\
&\leq
&\exp(-\tilde{c}'\delta^2N)+\frac{1}{45m}C''(p,q),\nonumber\end{eqnarray}where
$C''(p,q)$ is the same as (4) in Remark \ref{13002}.
 If
$\exp(\frac{\tilde{c}}{10}\delta^2 N)\leq m$, then
\begin{equation}
|L_{N'}(E)+L_N(E)-2L_{2N}(E)|\leq
\exp(-\hat{c}\delta^2N).\end{equation}
\end{proof}

Now,we can prove
\begin{lemma}\label{62002}Let $\tilde{c}$ be as in (\ref{25001}), $\hat{c}$ be as in Lemma \ref{62001}.
Assume that $L_{N_0}(E)>100\delta>0$ and
$\exp(-\hat{c}\delta^2N_0)\leq \frac{\delta}{12}$, where $\delta<1$
is a constant not depending on $N_0$, and
$L_{2N_0}(E)>\frac{9}{10}L_{N_0}(E)$. There exists
$\tilde{N}_0=\tilde{N}_0(p,q,\delta,N_0)\leq
(\exp(\frac{\tilde{c}}{8}\delta^2 N_0)+1)N_0$ such that for any
$N\geq \tilde{N}_0$ holds \[
|L_N(E)+L_{N_0}(E)-2L_{2N_0}(E)|<\exp(-\bar{c}'\delta^2 N_0),\]where
$\bar{c}'=\bar{c}'(p,q,\omega)$. Furthermore,
 \be |L(E)+L_{N_0}(E)-2L_{2N_0}(E)|< \exp(-\bar{c}\delta^2
 N_0),\ee where $\bar{c}=\bar{c}(p,q,\omega)$.
\end{lemma}
\begin{proof} We first prove the second part.
By lemma \ref{62001} for $N'_1=mN_0$,
$\exp(\frac{\tilde{c}}{8}\delta^2 N_0)\leq m<
\exp(\frac{\tilde{c}}{8}\delta^2 N_0)+1$, one has
\begin{equation}\label{54001}|L_{N'_1}(E)+L_{N_0}(E)-2L_{2N_0}(E)|<
\exp(-\hat{c}\delta^2 N_0).\end{equation} and
\[|L_{2N'_1}(E)+L_{N_0}(E)-2L_{2N_0}(E)|<
\exp(-\hat{c}\delta^2 N_0).\]In particular
\[|L_{N'_1}(E)-L_{2N'_1}(E)|< 2\exp(-\hat{c}\delta^2 N_0).\]
Since $0\leq L_{N_0}(E)-L_{2N_0}(E)<\frac{1}{10}L_{N_0}(E)$, one
obtains using (\ref{54001})
\[L_{N'_1}(E)>L_{N_0}(E)-2(L_{N_0}(E)-L_{2N_0}(E))-\exp(-\hat{c}\delta^2 N_0)>\frac{4}{5}L_{N_0}(E)-\exp(-\hat{c}\delta^2 N_0)>79\delta,\]
and
\[|L_{N'_1}(E)-L_{2N'_1}(E)|\leq 2\exp(-\hat{c}\delta^2 N_0)<2\delta<\frac{2}{79}L_{N'_1}(E)<\frac{1}{10}L_{N'_1}(E).\]
Set $\delta'=\frac{1}{2}\delta$, then $L_{N'_1}(E)>100 \delta'$, and
Lemma \ref{62001} applies for $N'_2=m_1N'_1$,
$\exp(\frac{\tilde{c}}{8}\delta'^2 N'_1)\leq m_1<
\exp(\frac{\tilde{c}}{8}\delta'^2 N'_1)+1$,
\[|L_{N'_2}(E)+L_{N'_1}(E)-2L_{2N'_1}(E)|\leq
\exp(-\hat{c}\delta'^2 N'_1).\] Also
\[L_{N'_2}(E)>L_{N'_1}(E)-2|L_{N'_1}(E)-L_{2N'_1}(E)|-\exp(-\hat{c}\delta'^2 N_1)>\frac{4}{5}L_{N_0}(E)-6\exp(-\hat{c}\delta^2
N_0)>79\delta>100\delta',\]
\[|L_{2N'_2}(E)+L_{N'_1}(E)-2L_{2N'_1}(E)|\leq
\exp(-\hat{c}\delta'^2 N'_1),\]
\[|L_{N'_2}(E)-L_{2N'_2}(E)|<2\exp(-\hat{c}\delta'^2 N'_1).\]
Since $N'_1>8N_0$ one has
\[\exp(-\hat{c}\delta'^2N'_1)=\exp(-\hat{c}\frac{\delta^2}{4}N'_1)<(\exp(-\hat{c}\delta^2
N_0))^2<(\frac{\delta}{12})^2.\] That implies, in particular
\[|L_{N'_2}(E)-L_{2N'_2}(E)|<2\exp(-\hat{c}\delta'^2
N'_1)<2\delta<\frac{1}{10}L_{N'_2}(E).\] Then Lemma \ref{62001}
applies for $N'_3=m_2N'_2$, $\exp(\frac{\tilde{c}}{8}\delta'^2
N'_2)\leq m_2<\exp(\frac{\tilde{c}}{8}\delta'^2 N'_2)+1$. E.T.C..
Obtain $N'_{i+1}=m_iN'_i$, $\exp(\frac{\tilde{c}}{8}\delta'^2
N'_i)\leq m_i<\exp(\frac{\tilde{c}}{8}\delta'^2 N'_i)+1$ with the
same $\delta'$. Then
\[|L_{N'_{i+1}}(E)+L_{N'_i}(E)-2L_{2N'_i}(E)|\leq
\exp(-\hat{c}\delta'^2 N'_i),\]
\[L_{N'_{i+1}}(E)>L_{N'_i}(E)-2|L_{N'_i}(E)-L_{2N'_i}(E)|-\exp(-\hat{c}\delta'^2 N'_i)>\frac{4}{5}L_{N_0}(E)-\sum_{j=1}^{i}(\frac{1}{2})^j\delta\geq 79\delta>50\delta=100\delta',\]
\[|L_{2N'_{i+1}}(E)+L_{N'_i}(E)-2L_{2N'_i}(E)|\leq
\exp(-\hat{c}\delta'^2 N'_i),\]
\[|L_{N'_{i+1}}(E)-L_{2N'_{i+1}}(E)|<2\exp(-\hat{c}\delta'^2 N'_i),\]
\[4\exp(-\hat{c}\delta'^2 N'_i)<(\frac{1}{2})^{i+1}\delta,\]and
\[|L_{N'_{i+1}}(E)-L_{2N'_{i+1}}(E)|<2\delta<\frac{1}{10}L_{N'_{i+1}}(E).\]
What's more,
\begin{eqnarray}\label{1002}
|L_{N'_{i+1}}(E)-L_{N'_{i}}(E)|&\leq&
|L_{N'_{i+1}}(E)+L_{N'_i}(E)-2L_{2N'_i}(E)|+2|L_{N'_i}(E)-L_{2N'_i}(E)|\\
&<&\exp(-\hat{c}\delta'^2 N'_i)+4\exp(-\hat{c}\delta'^2
N'_{i-1})<5\exp(-\hat{c}\delta'^2 N'_{i-1}),\ \ i\geq 2\nonumber
\end{eqnarray}
\[|L_{N'_2}(E)-L_{N'_1}(E)|<5\exp(-\hat{c}\delta^2 N_0).\] Since
$L_{N'_i}\to L(E)$ with $i\to \infty$ one has
\begin{eqnarray}\label{60214}\\
|L(E)+L_{N_0}(E)-2L_{2N_0}(E)|&=&|\sum_{i\geq
1}(L_{N'_{i+1}}(E)-L_{N'_{i}}(E))+L_{N'_1}(E)+L_{N_0}(E)-2L_{2N_0}(E)|\nonumber\\
&\leq &\sum_{s\geq 1}|L_{N'_{s+1}}(E)-L_{N'_{s}}(E)|+|L_{N'_1}(E)+L_{N_0}(E)-2L_{2N_0}(E)|\nonumber\\
&=&\sum_{s\geq
2}|L_{N'_{s+1}}(E)-L_{N'_{s}}(E)|+|L_{N'_2}(E)-L_{N'_1}(E)|+|L_{N'_1}(E)+L_{N_0}(E)-2L_{2N_0}(E)|\nonumber\\
&<&\sum_{s\geq 2}5\exp(-\hat{c}\delta'^2
N'_{i-1})+5\exp(-\hat{c}\delta^2
N_0)+\exp(-\hat{c}\delta^2 N_0)\nonumber\\
&<&\exp(-\bar{c}\delta^2 N_0).\nonumber\end{eqnarray} That finishes
we second part. We prove now the first part. Note that just as in
(\ref{60214}) one obtains
\be\label{60222}|L_{N'_i}(E)+L_{N_0}(E)-2L_{N_0}(E)|\leq
\exp(-\bar{c}\delta^2 N_0)\ee for $i\geq 1$. Let $N\geq
\tilde{N}_0:=N'_1$ be arbitrary. Find $i$ such that $N'_i\leq
N<N'_{i+1}$. Recall that
\[N'_i=m_{i-1}N'_{i-1},\ \exp(\frac{\tilde{c}}{8}\delta'^2
N'_{i-1})\leq m_{i-1}<\exp(\frac{\tilde{c}}{8}\delta'^2
N'_{i-1})+1,\]
\[N'_{i+1}=m_{i}N'_{i},\ \exp(\frac{\tilde{c}}{8}\delta'^2
N'_{i})\leq m_{i}<\exp(\frac{\tilde{c}}{8}\delta'^2 N'_{i})+1,\]here
$N'_0:=N_0$ for convenience. Consider two cases:\begin{enumerate}
\item[{\rm{(1)}}]$N\leq \exp(\frac{\tilde{c}}{4}\delta'^2
N'_{i-1})N'_{i-1}$. In this case $\frac{N'_{i-1}}{N}\leq
\frac{N'_{i-1}}{N'_i}\leq \exp(-\frac{\tilde{c}}{8}\delta'^2
N'_{i-1})$. Then find $\tilde{m}$, $m_{i-1}\leq \tilde{m}\leq
\exp(\frac{\tilde{c}}{4}\delta'^2 N'_{i-1})$, such that
\[\tilde{m}N'_{i-1}\leq N<(\tilde{m}+1)N'_{i-1}.\] Then by Lemma
\ref{62001}
\be\label{60216}|L_{\tilde{m}N'_{i-1}}(E)+L_{N'_{i-1}}(E)-2L_{2N'_{i-1}}(E)|<\exp(-\hat{c}\delta'^2
N'_{i-1}).\ee Note that \[N-\tilde{m}N'_{i-1}\leq N'_{i-1},\] and by
Remark \ref{23023}
\begin{eqnarray}&&\left |\log\|\tilde{M}_N(x,E)\|-\log\|\tilde{M}_{\tilde{m}N'_{i-1}}(x,E)\|\right
|\leq
\log\|\tilde{M}_{N-\tilde{m}N'_{i-1}}(x+\tilde{m}N'_{i-1}\omega,E)\|\nonumber\\
&&\ \ \ \ \ \ \leq N'_{i-1}(\log
C(p,q)-D)-(N-\tilde{m}N'_{i-1})F_{N-\tilde{m}N'_{i-1}}(x+\tilde{m}N'_{i-1}\omega).\nonumber
\end{eqnarray}
We know that
\[\mes(\{x:|kF_k(x)-k<F_k(x)>|>k\delta\})<\exp(-c\delta k)\]for any
$k$. Since $<F_k>=0$, then
\begin{eqnarray}\label{60018}\left |\log\|\tilde{M}_N(x,E)\|-\log\|\tilde{M}_{\tilde{m}N'_{i-1}}(x,E)\|\right
|&\leq& N'_{i-1}(\log
C(p,q)-D)+N'_{i-1}\\&<&\hat{C}'(p,q)N'_{i-1},\nonumber\end{eqnarray}
if $x\not\in\hat{\mathbb{B}}$,
$\mes{\hat{\mathbb{B}}}<\exp(-cN'_{i-1})$. Integrating (\ref{60018})
over $\mathbb{T}\backslash \hat{\mathbb{B}}$ and using
(\ref{11111}), one obtains
\begin{eqnarray}|L_N(E)-\frac{\tilde{m}N'_{i-1}}{N}L_{\tilde{m}N'_{i-1}}(E)|&<&
\hat{C}'(p,q)\frac{N'_{i-1}}{N}+2\tilde{C}(p,q)^{\frac{1}{2}}\times\exp(-\frac{c}{2}N'_{i-1})\\
&\leq&\hat{C}'(p,q)\exp(-\frac{\tilde{c}}{8}\delta'^2
N'_{i-1})+2\tilde{C}(p,q)^{\frac{1}{2}}\times\exp(-\frac{c}{2}N'_{i-1})\nonumber\\
&\leq& \exp(-\tilde{c}_1\delta'^2 N'_{i-1})\nonumber.\end{eqnarray}
Note that
\[1-\frac{\tilde{m}N'_{i-1}}{N}=\frac{N-\tilde{m}N'_{i-1}}{N}\leq
\frac{N'_{i-1}}{N}\leq \exp(-\frac{\tilde{c}}{8}\delta'^2
N'_{i-1}).\] Thus
\be\label{60219}|L_N(E)-L_{\tilde{m}N'_{i-1}}(E)|\leq
\exp(-\tilde{c}_2\delta'^2 N'_{i-1}).\ee Combining (\ref{60216})
with (\ref{60219}), one obtains
\be\label{60223}|L_N(E)+L_{N'_{i-1}}(E)-2L_{2N'_{i-1}}(E)|<\exp(-\tilde{c}_3\delta^2
N'_{i-1}).\ee
\item[{\rm{(2)}}]$N>\exp(\frac{\tilde{c}}{4}\delta'^2N'_{i-1})N'_{i-1}$.
Find $\tilde{m}'$ such that \[\tilde{m}'N'_i\leq
N<(\tilde{m}'+1)N'_i.\] Thus
\[\tilde{m}'>\frac{\exp(\frac{\tilde{c}}{4}\delta'^2
N'_{i-1})}{\exp(\frac{\tilde{c}}{8}\delta'^2
N'_{i-1})+1}<\frac{\exp(\frac{\tilde{c}}{4}\delta'^2
N'_{i-1})}{2\exp(\frac{\tilde{c}}{8}\delta'^2
N'_{i-1})}=\frac{1}{2}\exp(\frac{\tilde{c}}{8}\delta'^2
N'_{i-1}).\]Since $N<N'_{i+1}$, then $\tilde{m}'<m_i$. It implies
due to Lemma \ref{62002}
\be\label{60220}|L_{\tilde{m}'N'_{i}}(E)+L_{N'_{i}}(E)-2L_{2N'_{i}}(E)|<\exp(-\tilde{c}'\delta'^2
N'_{i})+\frac{2}{9\tilde{m}'}L_{N'_i}(E).\ee As in Case (1), one has
\be\label{60221}|L_N(E)-L_{\tilde{m}N'_{i-1}}(E)|<\frac{\hat{C}''(p,q)}{\tilde{m}'}\leq
\exp(-\tilde{c}_4\delta'^2 N'_{i-1}).\ee Combining (\ref{60220})
with (\ref{60221}), one obtains
\be\label{60224}|L_N(E)+L_{N'_{i}}(E)-2L_{2N'_{i}}(E)|<\exp(-\tilde{c}_5\delta'^2
N'_{i-1}).\ee
\end{enumerate}
Combining (\ref{60222}) with (\ref{60223}) or (\ref{60224}), as in
(\ref{60214}), one obtains
\[|L_N(E)+L_{N_0}(E)-2L_{N_0}(E)|<\exp(-\bar{c}'\delta^2N_0),\]
where $\bar{c}'=\bar{c}'(p,q,\omega)$.
\end{proof}

\begin{lemma}\label{62003}Assume $L(E_0)>0$. There exists
$\check{C}(p,q,E_0)$ such that with
$\rho_0'(E_0,N)=\frac{L(E_0)}{200}\exp(-\check{C}(p,q,E_0)N)$, one
has
\[|L_{N}(E_0)-L_{N}(E)|<\frac{L(E_0)}{100},\]for any
$|E-E_0|<\rho_0'(E_0,N)$ and any $N$.\end{lemma}
\begin{proof}
Note that
\begin{eqnarray}
\left |\|T_{N}(x,E_0)\|-\|T_{N}(x,E)\|\right |
&\leq&\|T_{N}(x,E_0)-T_{N}(x,E)\|\\&\leq&\sum_{j=0}^{N-1}(\|B(x+(N-1)\omega,E_0)\times\cdots\times
B(x+(j+1)\omega,E_0)\|\times \nonumber\\&\ &\ \ \ \
\|B(x+j\omega,E_0)-
B(x+j\omega,E)\|\times\|B(x+(j-1)\omega,E)\times\cdots\times
B(x,E)\|)
\nonumber\\
&\leq&NC(p,q)^{N-1}|E_0-E|\nonumber,
\end{eqnarray}
see (1) in Lemma \ref{lem12002}. By (\ref{16001}), one has
\begin{eqnarray}\left
|\|\tilde{M}_{N}(x,E_0)\|-\|\tilde{M}_{N}(x,E)\|\right
|&=&\prod_{n=0}^{N-1}\frac{1}{|q(x+n\omega)q(x+(n+1)\omega)|^{\frac{1}{2}}}\big
|\|T_{N}(x,E_0)\|-\|T_{N}(x,E)\|\big |\\
&\leq
&\frac{NC(p,q)^{N-1}|E_0-E|}{\prod_{n=0}^{N-1}|q(x+n\omega)q(x+(n+1)\omega)|^{\frac{1}{2}}}\nonumber\end{eqnarray}
Assume for instance that $\|\tilde{M}_{N}(x,E_0)\|\geq
\|\tilde{M}_{N}(x,E)\|$. Then
\begin{eqnarray}\label{60002}\left |\log\|\tilde{M}_{N}(x,E_0)\|-\log \|\tilde{M}_{N}(x,E)\| \right |&=&
\log\frac{\|\tilde{M}_{N}(x,E_0)\|}{\|\tilde{M}_{N}(x,E)\|}=\log
(1+\frac{\|\tilde{M}_{N}(x,E_0)\|-
\|\tilde{M}_{N}(x,E)\|}{\|\tilde{M}_{N}(x,E)\|})\\
&\leq &\frac{\|\tilde{M}_{N}(x,E_0)\|-
\|\tilde{M}_{N}(x,E)\|}{\|\tilde{M}_{N}(x,E)\|}\leq\|\tilde{M}_{N}(x,E_0)\|-
\|\tilde{M}_{N}(x,E)\|\nonumber\\
&\leq&\frac{NC(p,q)^{N-1}|E_0-E|}{\prod_{n=0}^{N-1}|q(x+n\omega)q(x+(n+1)\omega)|^{\frac{1}{2}}}.\nonumber\end{eqnarray}
Due to  (1) in Remark \ref{23006} for any $\delta$ and any $K$
\[\mes(\{x:|\sum_{k=1}^K\frac{1}{2}\log
|q(x+k\omega)q(x+(k+1)\omega)|-K<\frac{1}{2}\log
|q(x)q(x+\omega)|>|>\delta K\})<\exp(-c_q\delta K).
\]Thus
\begin{eqnarray}\left |\sum_{k=0}^{N-1}\frac{1}{2}\log
|q(x+k\omega)q(x+(k+1)\omega)|\right |&<& N|<\frac{1}{2}\log
|q(x)q(x+\omega)|>|+\frac{800\tilde{C}(p,q)^{\frac{1}{2}}}{L(E_0)c_q}N\\&=&|D|N+\frac{800\tilde{C}(p,q)^{\frac{1}{2}}}{L(E_0)c_q}N=\hat{C}(q,p,E_0)N,\nonumber\end{eqnarray}
if $x\not\in\mathbb{B}_1,\ \mes{\mathbb{B}_1}<\exp(-c_q\times
\frac{800\tilde{C}(p,q)^{\frac{1}{2}}}{L(E_0)c_q}N)=\exp(-\frac{800\tilde{C}(p,q)^{\frac{1}{2}}}{L(E_0)}N),$
where constant $\tilde{C}(p,q)$ comes from (4) in Remark
\ref{13002}. The same estimate holds if
$\|\tilde{M}_{N}(x,E_0)\|\leq \|\tilde{M}_{N}(x,E)\|$. So
\begin{equation}\left |\log\|\tilde{M}_{N}(x,E_0)\|-\log
\|\tilde{M}_{N}(x,E)\| \right |\leq
NC(p,q)^{N-1}|E_0-E|\exp(\hat{C}(p,q,E_0)N)\leq
\exp(\check{C}(p,q,E_0)N)|E_0-E|,\end{equation}if
$x\not\in\mathbb{B}_1,\
\mes{\mathbb{B}_1}<\exp(-\frac{800\tilde{C}(p,q)^{\frac{1}{2}}}{L(E_0)}N).$
Set $\rho_0'=\frac{L(E_0)}{200}\exp(-\check{C}(p,q,E_0)N)$. Then, if
$|E-E_0|\leq \rho_0'$,
\[\left
|\log\|\tilde{M}_{N}(x,E_0)\|-\log \|\tilde{M}_{N}(x,E)\| \right
|<\frac{L(E_0)}{200},\]if $x\not\in\mathbb{B}_1,\
\mes{\mathbb{B}_1}<\exp(-\frac{800\tilde{C}(p,q)^{\frac{1}{2}}}{L(E_0)}N),$
 \be\label{60012}\left |\int_{\mathbb{T}\backslash
\mathbb{B}_1}\log\|\tilde{M}_{N}(x,E_0)\|-\int_{\mathbb{T}\backslash
\mathbb{B}_1}\log \|\tilde{M}_{N}(x,E)\| \right
|<\frac{L(E_0)}{200}.\ee  Due to (\ref{11111}),
\[|\int_{\mathbb{B}_1}\tilde{u}_{N}dx|\leq
\tilde{C}(p,q)^{\frac{1}{2}}\exp(-\frac{400\tilde{C}(p,q)^{\frac{1}{2}}}{L(E_0)}N)\]
for $E $ or $E_0$. As $y\exp(-\xi y)\leq \xi^{-1}$ for any $y,
\xi>0$. Thus
\be\label{60013}|\int_{\mathbb{B}_1}\tilde{u}_{N}dx|\leq
\frac{L(E_0)}{400N}\leq \frac{L(E_0)}{400}\ee for $E $ or $E_0$.
Combining (\ref{60012}) with (\ref{60013}), one has
\[|L_{N}(E_0)-L_{N}(E)|<\frac{L(E_0)}{200N}+2\frac{L(E_0)}{400}\leq\frac{L(E_0)}{100}.\]
\end{proof}

\begin{lemma}\label{62006}Assume $L(E_0)>0$. There exists $\rho'_0=\rho'_0(p,q,E_0,\omega)>0$ and $\tilde{N}_0=\tilde{N}_0(p,q,E_0,\omega)<+\infty$ such
that for any $N\geq\tilde{N}_0$ and any $|E-E_0|<\rho'_{0}$
 \[|L_{N}(E)-L(E)|<\frac{1}{20}L(E),\ \
\frac{11}{10}L(E_0)>L(E)>\frac{9}{10}L(E_0).\]
\end{lemma}
\begin{proof}
One has $\lim_{n\to\infty}L(E_0)=L(E_0)$. Therefore there exists
$N_0=N_0(p,q,\omega,E_0)$ s.t.
$|L_n(E_0)-L(E_0)|<\frac{L(E_0)}{100}$ for $n\geq
N_0(p,q,\omega,E_0)$. It implies that
$L_{N_0}(E_0)-L_{2N_0}(E_0)<\frac{L(E_0)}{100}$, as $L(E_0)\leq
L_{2N_0}(E_0)\leq L_{N_0}(E_0)$. Set
$\delta=\min(\frac{1}{200}L(E_0),\frac{1}{2})$. We can assume also
that $\exp(-\hat{c}\delta^2N_0)\leq \frac{\delta}{12}$,
$\exp(-\bar{c}\delta^2
 N_0))<\frac{1}{50}L(E_0)$ and $\exp(-\bar{c}'\delta^2
 N_0))<\frac{1}{50}L(E_0)$, where $\hat{c}$ is  as in Lemma \ref{62001},
$\bar{c}$ and $\bar{c}'$ are as in Lemma \ref{62002}. Using Lemma
\ref{62003} applied to $N_0$ and $2N_0$. One has for
$|E-E_0|<\rho'_0(E_0,2N_0)$ \be\label{60033}L_{N_0}(E)\geq L(E_0)-
|L_{N_0}(E)-L_{N_0}(E_0)|-|L_{N_0}(E_0)-L(E_0)|>
L(E_0)-\frac{L(E_0)}{100}-\frac{L(E_0)}{100}=\frac{49}{50}L(E_0),\ee
and
\begin{eqnarray}\label{60034}|L_{N_0}(E)-L_{2N_0}(E)|&\leq&
|L_{N_0}(E)-L_{N_0}(E_0)|+|L_{N_0}(E_0)-L_{2N_0}(E_0)|+|L_{2N_0}(E_0)-L_{2N_0}(E)|\\
&<&\frac{L(E_0)}{100}+\frac{L(E_0)}{100}+\frac{L(E_0)}{100}=\frac{3}{100}L(E_0)<\frac{1}{10}L_{N_0}(E).\nonumber\end{eqnarray}
 Thus Lemma \ref{62002}
applies for $L_{N_0}(E)$, $\delta$, $N_0$ and $E$, then there exists
$\tilde{N}_0=\tilde{N}_0(p,q,\delta,N_0)\leq
(\exp(\frac{\tilde{c}}{8}\delta^2 N_0)+1)N_0$ such that for any
$N\geq \tilde{N}_0$ holds \[
|L_N(E)+L_{N_0}(E)-2L_{2N_0}(E)|<\exp(-\bar{c}'\delta^2 N_0),\] and
 \be\label{60035} |L(E)+L_{N_0}(E)-2L_{2N_0}(E)|< \exp(-\bar{c}\delta^2
 N_0),\ee where
$\bar{c}'=\bar{c}'(p,q,\omega)$ and $\bar{c}=\bar{c}(p,q,\omega)$
are as in Lemma \ref{62003}. These imply
\begin{eqnarray}\label{60036}|L(E)-L_{N}(E)|&\leq&\exp(-\bar{c}'\delta^2 N_0)+\exp(-\bar{c}\delta^2
 N_0)<
 \frac{1}{50}L(E_0)+\frac{1}{50}L(E_0)=\frac{1}{25}L(E_0).
 \end{eqnarray}
Combining (\ref{60033}), (\ref{60034}) with (\ref{60035}), one
obtains
\begin{eqnarray}|L(E_0)-L(E)|&\leq &
|L(E)+L_{\tilde{N}_0}(E)-2L_{2\tilde{N}_0}(E)|+|L(E_0)-L_{\tilde{N}_0}(E)|+2|L_{\tilde{N}_0}(E)-L_{2\tilde{N}_0}(E)|\\&<&
\frac{1}{50}L(E_0)+\frac{1}{50}L(E_0)+2\frac{3}{100}L(E_0)=\frac{1}{10}L(E_0).\nonumber\end{eqnarray}It
implies
\be\label{60038}\frac{11}{10}L(E_0)>L(E)>\frac{9}{10}L(E_0),\ee and
\[|L(E)-L_{N}(E)|<\frac{1}{25}L(E_0)<\frac{1}{25}\times
\frac{10}{9}L(E)=\frac{2}{45}L(E)<\frac{1}{20}L(E).\]
\end{proof}

\begin{lemma}\label{22011}Assume $L(E_0)>0$.  Let $\rho'_0$ be as in Lemma
\ref{62006}. Fix any $0<\kappa <1$. Let $K=\frac{\kappa}{20} N$.
Then for there exists $N_1$ s.t. $N\geq N_1$ and  for any
$E\in(E_0-\rho'_0,E_0+\rho'_0)$ holds
\[\mes\{x:|\tilde{u}_N(x,E)-\frac{1}{K}\sum_{k=1}^K\tilde{u}_N(x+k\omega,E)|>\kappa
L(E)\}<\exp(-c''\kappa L(E)N),\]where  constant $c''$ depends only
on $p,q,\omega $, but does not depend on $E_0$ or $E$ or $\kappa$.
The number $N_1$ depends on $p,q,\omega$,$E_0$ and
$\kappa$.\end{lemma}
\begin{proof}
Choose $\bar{N}_0$ s.t. \be\label{21118}
\bar{N}_0>\max(\frac{800(\log C(p,q)-D)}{\kappa^2
L(E_0)},\frac{20C_6}{\kappa^2
L(E_0)},20,\frac{1}{\kappa},\tilde{N}_0)\ee where $(\log C(p,q)-D)$
is as in Remark \ref{23023}, $C_6$ is as in Lemma \ref{22010},
$\tilde{N}_0$ is as in Lemma \ref{62006}. Thus
 \be\label{21119} L_{K}(E)<(1+\frac{1}{20})L(E),\ee for any $K\geq \bar{N}_0$ and any $E\in(E_0-\rho'_0,E_0+\rho'_0)$. Finally,
 we assume that \be\label{21120}\log K<K^{\frac{1}{6}},\ee
if $K\geq \bar{N}_0$. Using Lemma \ref{22021} and Remark \ref{23023}
one obtains
\begin{eqnarray}\label{21124}|\tilde{u}_N(x,E)-\frac{1}{K}\sum_{k=1}^K\tilde{u}_N(x+k\omega,E)|&\leq&
\frac{1}{KN}\big [ \sum_{k=1}^{\bar{N}_0}2k(\log
C(p,q)-D)+\sum_{k=\bar{N}_0+1}^K 2 kL_k(E)+\sum_{k=\bar{N}_0+1}^K 2C_6(k\log k)^{\frac{1}{2}}\big ]\nonumber\\
&&\ \ \ \ -\frac{1}{KN}\sum_{k=1}^K(kF_k(x)+kF_k(x+N\omega))\nonumber\\
&=&(\Rmnum{1})+(\Rmnum{2})\nonumber.
\end{eqnarray}
Take here $N\geq \bar{N}_0^3$, $K=\frac{\kappa}{20}N$. Note that
$N_0\kappa>1$, so $K=\frac{\kappa}{20}N>\frac{\bar{N}_0^2}{20}\geq
\bar{N}_0$. Thus, (\ref{21119}) and (\ref{21120}) holds.
 One
has\begin{eqnarray}\label{21122}(\Rmnum{1})&<&\frac{\bar{N}_0^2(\log C(p,q)-D)}{KN}+4\frac{21}{20}L(E)\frac{K}{N}++\frac{C_6K^{\frac{1}{2}}(\log K)^{\frac{1}{2}}}{N}\\
&<& \frac{\kappa L(E_0)}{20}+\frac{\kappa}{4}L(E)+\frac{\kappa
 L(E_0)}{20}<\frac{1}{2}\kappa L(E),\nonumber\end{eqnarray} see
(\ref{21118}), (\ref{21119}), (\ref{21120}) and Lemma \ref{62006}.
If
\[\sum_{k=1}^KkF_k(x,E)<-\frac{KN}{4}\kappa L(E),\]then
\[\exists k,\ s.t. \ kF_k(x,E)<-\frac{N}{4}\kappa L(E).\]
We know that
\[\mes(\{x:|kF_k(x)-k<F_k(x)>|>k\delta\})<\exp(-c\delta k),\]
Since $<F_k>=0$, then
\begin{eqnarray}\mes(\{x:kF_k(x)<-\frac{N}{4}\kappa
L(E)\})&\leq&\mes(\{x:|kF_k(x)|>\frac{N}{4}\kappa L(E)
<\exp(-c\frac{N\kappa L(E)}{4k}k)=\exp(-c_2\kappa N
L(E)).\nonumber\end{eqnarray} So
\[\mes\{x:\sum_{k=1}^K kF_k(x)<-\frac{KN}{4}L(E)\}\leq K\exp(-c_2\kappa N
L(E))=K\exp(-c_2 20KL(E)).\] Sine $y\exp(-\xi y)\leq \xi^{-1}$ for
any $\xi$, $y>0$, one has \begin{eqnarray}\label{21123}
\mes\{x:\sum_{k=1}^K kF_k(x)<-\frac{KN}{4}L(E)\}&\leq &K\exp(-c_2
20KL(E))=K\exp(-c_2 10KL(E))\exp(-c_2 10KL(E))\\
&\leq & \frac{1}{10c_2L(E)}\exp(-c_2
10KL(E))=\frac{1}{10c_2L(E)}\exp(-\frac{c_2}{2}\kappa L(E)N)\nonumber\\
&\leq& \exp(-c'_2\kappa L(E) N),\nonumber\end{eqnarray} if $N$ is
large enough depending on $L(E_0)$ and $\kappa$ (see Lemma
\ref{62006} and (\ref{21118})). Combining (\ref{21122}) with
(\ref{21123}) one has
\[\mes\{x:|\tilde{u}_N(x,E)-\frac{1}{K}\sum_{k=1}^K\tilde{u}_N(x+j\omega,E)|>\kappa
L(E)\}\leq 2\exp(-c'_2\kappa N L(E))<\exp(-c''\kappa L(E)N),\]where
constant $c''$ depends only on $p,q,\omega.$
\end{proof}

\begin{lemma}\label{22012}Assume $L(E_0)>0$.  Let  $\rho'_0$ be as in Lemma
\ref{62006}, $N_1$ be as in Lemma \ref{22011} with
$\kappa=\frac{1}{40}$. Then for $N\geq N_1$ and any
$E\in(E_0-\rho'_0,E_0+\rho'_0)$ holds
\[\mes\{x:|\tilde{u}_N(x,E)-L(E)|> \frac{L(E)}{10}\}<\exp(-c L(E) N),\]  where constant $c$ depends only
on $p,q,\omega$, but does not depend on $E$ or $E_0$.
\end{lemma}
\begin{proof}Due to Remark \ref{23006} for $K>\check{K}$ holds
\[\mes\{x:|\sum_{k=1}^K\tilde{u}_N(x+k\omega,E)-K<\tilde{u}_N(\cdot,E)>|>\delta
K\}\leq \exp(-\hat{c}\delta K),\] where
$\hat{c}=\hat{c}(p,q,\omega)$. Set $\delta=\frac{ L(E)}{40}$. Thus
$\check{K}=\check{K}(p,q,\omega, E_0)$. Due to Lemma \ref{62006},
$L_N(E)<(1+\frac{1}{20})L(E)$ if $N\geq \tilde{N}_0$, where
$\tilde{N}_0$ is as in Lemma \ref{62006}. Note that if
\[|\frac{1}{K}\sum_{k=1}^K\tilde{u}_N(x+k\omega,E)-<\tilde{u}_N(\cdot,E)>|\leq\delta,\]then
\begin{eqnarray}|\frac{1}{K}\sum_{j=1}^K\tilde{u}_N(x+j\omega,E)-L(E)|&\leq
&
|\frac{1}{K}\sum_{j=1}^K\tilde{u}_N(x+j\omega,E)-L_N(E)|+|L_N(E)-L(E)|<\delta+\frac{1}{20}L(E)\nonumber\\
&=&\frac{1}{40} L(E)+\frac{1}{20} L(E) =\frac{3}{40}
L(E)\nonumber\end{eqnarray} Therefore
\[\mes\{x:|\frac{1}{K}\sum_{j=1}^K\tilde{u}_N(x+j\omega,E)-L(E)|>\frac{3}{40}
L(E)\}<\exp(-\frac{\hat{c}}{40}  L(E) K).\] Let $K=\frac{ N}{800}$
as in Lemma \ref{22011}, with $\kappa=\frac{1}{40}$. Then for $N\geq
N_1$ holds
\[\mes\{x:|\tilde{u}_N(x,E)-\frac{1}{K}\sum_{k=1}^K\tilde{u}_N(x+k\omega,E)|>\frac{ L(E)}{40}\}\leq
\exp(-\frac{c''}{40} L(E)N).\] Recall that $N_1>\tilde{N}_0$ with
$K=\frac{N}{800}$(see (\ref{21118})). Let $N\geq N_1$. Then
\[\mes\{x:|\tilde{u}_N(x,E)-L(E)|>\frac{ L(E)}{10}\}<\exp(-\frac{\hat{c}}{32000} L(E)
N)+\exp(-\frac{c''}{40} L(E)N)<\exp(-c_4 L(E) N),\] $c_4$ depends
only on $p,q, \omega$. Here we replace  $c_4$ by $c$ for convenient
notations.
\end{proof}

\begin{lemma}\label{32002}
Assume $L(E_0)>0$.  Let  $\rho'_0$ be as in Lemma \ref{62006}, $N_1$
be as in Lemma \ref{22011} with $\kappa=\frac{1}{40}$, $c$ be as in
lemma \ref{22012}. Let $N\geq N_1$ and
$E\in(E_0-\rho'_0,E_0+\rho'_0)$ be arbitrary. Let $N'=mN, m\in
\mathbb{N}$ and $\exp(\frac{c}{10}L(E) N)\leq m\leq
\exp(\frac{c}{4}L(E) N) $. Then \be
|L_{N'}(E)+L_{N}(E)-2L_{2N}(E)|\leq \exp(-c_6 L(E) N),\ee where
$c_6=c_6(p,q,\omega)$.
\end{lemma}
\begin{proof}
Let $\mathbb{G}_j=\{x:|\tilde{u}_N(x+j\omega N)-L(E)|\leq
\frac{1}{10} L(E)\}\bigcap\{x:|\tilde{u}_{2N}(x+j\omega N)-L(E)|\leq
\frac{1}{10} L(E)\}.$ By Lemma \ref{22012},
$\mes(\mathbb{T}\backslash\mathbb{G}_j)\leq 2\exp(-cL(E)N)$ for any
j. Set $\mathbb{G}=\bigcap_{0\leq j\leq m-1}\mathbb{G}_j$. Then
$\mes(\mathbb{T}\backslash \mathbb{G})\leq
\exp(-\frac{2c}{3}L(E)N)$. One has
\[\|\tilde{M}_{N}(x+jN\omega,E)\|>\exp(\frac{9}{10}NL(E)),\]
and
\begin{eqnarray}
&&|\log\|\tilde{M}_{N}(x+jN\omega,E)\|+\log\|\tilde{M}_{N}(x+(j+1)N\omega,E)\|-\log\|\tilde{M}_{N}(x+jN\omega,E)\tilde{M}_{N}(x+(j+1)N\omega,E)\||\\
&\ &\ \ \ \ \ \ \ \
<\frac{1}{10}NL(E)+\frac{1}{10}NL(E)+\frac{1}{10}2NL(E)
=\frac{2}{5}NL(E),\nonumber\end{eqnarray} for any $x\in\mathbb{G}$,
$0\leq j\leq m$. One has \[ \tilde{M}_{N'}(x,E)=\prod_{j=m}^1
\tilde{M}_N(x+(j-1)N\omega,E).\] If $x\in\mathbb{G}$, the avalanche
principle applies with $\mu=\exp(\frac{9}{10}NL(E))$. So
\[\left |\log
\|\tilde{M}_{N'}(x,E)\|+\sum_{j=2}^{m-1}\log
\|\tilde{M}_{N}(x+(j-1)N\omega,E)\|-\sum_{j=1}^{m-1}\log
\|\tilde{M}_{2N}(x+(j-1)N\omega,E)\|\right |\leq
C\frac{m}{\mu}=Cm\exp(-\frac{9}{10}NL(E)).\] Dividing by $N'$ and
integrating over $\mathbb{G}$, one obtains
\begin{eqnarray}\label{30102}
&&|\int_{\mathbb{G}}\tilde{u}_{N'}(x,E)dx+\frac{1}{m}\int_{\mathbb{G}}\sum_{j=2}^{m-1}\tilde{u}_{N}(x+(j-1)N\omega,E)dx
-\frac{2}{m}\int_{\mathbb{G}}\sum_{j=1}^{m-2}
\tilde{u}_{2N}(x+(j-1)N\omega,E)dx|\\&\leq&
C\frac{m}{N'}\exp(-\frac{9}{10}NL_N(E))\leq
C\exp(-\frac{9}{10}NL_N(E)) \nonumber.\end{eqnarray} Due to
(\ref{11111})
\[|\int_{\mathbb{T}\backslash\mathbb{G}}\tilde{u}_Kdx|\leq
\tilde{C}(p,q)^{\frac{1}{2}}\exp(-\frac{c}{3}L(E)N)\] for any $K$.
Thus
\begin{equation}\label{30101}|\int_{\mathbb{T}\backslash\mathbb{G}}\tilde{u}_{N'}(x,E)dx+
\frac{1}{m}\int_{\mathbb{T}\backslash\mathbb{G}}\sum_{j=2}^{m-1}\tilde{u}_{N}(x+jN\omega,E)dx
-\frac{2}{m}\int_{\mathbb{T}\backslash\mathbb{G}}\sum_{j=1}^{m-1}
\tilde{u}_{2N}(x+jN\omega,E)dx|\leq
4\tilde{C}(p,q)^{\frac{1}{2}}\exp(-\frac{c}{3}L(E)N).
\end{equation}
Combining (\ref{30101}) with (\ref{30102}), one has
\begin{equation}
|L_{N'}(E)+\frac{m-2}{m}L_N(E)-\frac{2(m-1)}{m}L_{2N}(E)| \leq
4\tilde{C}(p,q)^{\frac{1}{2}}\exp(-\frac{c}{3}L(E)N)+C\exp(-\frac{9}{10}NL(E))\leq
\exp(-c_5L(E) N)\nonumber.\end{equation} As $\exp(\frac{c}{10}L(E)
N)\leq m$, then
\begin{equation}
|L_{N'}(E)+L_N(E)-2L_{2N}(E)|\leq \exp(-c_6 L(E)
N),\end{equation}where $c_6=c_6(p,q,\omega)$.
\end{proof}

\begin{lemma}\label{62010}
Assume $L(E_0)>0$.  Let  $\rho'_0$ be as in Lemma \ref{62006}, $N_1$
be as in Lemma \ref{22011} with $\kappa=\frac{1}{40}$. Let $N\geq
N_1$ and $E\in(E_0-\rho'_0,E_0+\rho'_0)$ be arbitrary. Then
\begin{equation}
|L(E)+L_{N}(E)-2L_{2N}(E)|<\exp(-c L(E) N),
\end{equation}where $c=c(p,q,\omega)$.
\end{lemma}
\begin{proof} By Lemma \ref{32002}  for  $N'=mN$, with  $m\in
\mathbb{N}$, $\exp(\frac{c}{8}NL(E))\leq m<
\exp(\frac{c}{8}NL(E))+1$, one has
\[ |L_{N'}(E)+L_{N}(E)-2L_{2N}(E)|< \exp(-c_6L(E) N),
 \]and
\begin{displaymath}
|L_{2N'}(E)+L_N(E)-2L_{2N}(E)|<\exp(-c_6L(E) N).
\end{displaymath}
In particular
\begin{displaymath}
|L_{N'}(E)-L_{2N'}(E)|<2\exp(-c_6L(E) N).
\end{displaymath}
Pick $\exp(\frac{c}{8}NL(E))\leq m_1 <\exp(\frac{c}{8}NL(E))+1$. Set
$N'_1=m_1 N$. Similarly define $N'_2=m_2N'_1$, i.e. with $N'_1$ in
the role of $N$. E.T.C., obtain $N'_s$ such that
\[
|L_{N'_{s+1}}(E)+L_{N'_s}(E)-2L_{2N'_s}(E)|<\exp(-c_6 L(E) N'_s),
\]
\[
|L_{2N'_{s+1}}(E)+L_{N'_s}(E)-2L_{2N'_s}(E)|<\exp(-c_6L(E) N'_s),
\]
\[|L_{N'_{s+1}}(E)-L_{2N'_{s+1}}(E)|<2\exp(-c_6 L(E)
N'_s).\] One has $L_{N'_s}(E)\to L(E)$ with $s\to\infty$. So
\begin{eqnarray}\\
|L(E)+L_N(E)-2L_{2N}(E)|&=&|\sum_{s\geq
1}(L_{N'_{s+1}}(E)-L_{N'_{s}}(E))+L_{N'_1}(E)+L_N(E)-2L_{2N}(E)|\nonumber\\
&\leq &\sum_{s\geq 1}|L_{N'_{s+1}}(E)-L_{N'_{s}}(E)|+|L_{N'_1}(E)+L_N(E)-2L_{2N}(E)|\nonumber\\
&<&\sum_{s\geq 2}5\exp(-c_6L(E) N'_{s-1})+5\exp(-c_6L(E)
N)+\exp(-c_6L(E) N)<\exp(-c_7L(E)N).\nonumber
\end{eqnarray}Here again we change $c_7$ by c for convenient
notations.
\end{proof}

\section{proof of the main theorem }
\begin{proof}[Proof of Theorem \ref{mainthm}]Assume $\|\tilde{M}_N(x,E)\|\geq
\|\tilde{M}_N(x,E')\|$, then
\begin{eqnarray}\label{41002}\left |\log\|\tilde{M}_N(x,E)\|-\log \|\tilde{M}_N(x,E')\| \right |&=&
\log\frac{\|\tilde{M}_N(x,E,\omega)\|}{\|\tilde{M}_N(x,E',\omega)\|}=\log
(1+\frac{\|\tilde{M}_N(x,E)\|-
\|\tilde{M}_N(x,E')\|}{\|\tilde{M}_N(x,E)\|})\\
&\leq &\frac{\|\tilde{M}_N(x,E)\|-
\|\tilde{M}_N(x,E')\|}{\|\tilde{M}_N(x,E)\|}\leq\|\tilde{M}_N(x,E)\|-
\|\tilde{M}_N(x,E')\|\nonumber\\
&=&\frac{\|T_N(x,E)\|-
\|T_N(x,E')\|}{\prod_{n=0}^{N-1}|q(x+n\omega)q(x+(n+1)\omega)|^{\frac{1}{2}}}\nonumber\\&\leq&
\frac{\|T_N(x,E)-T_N(x,E')\|}{\prod_{n=0}^{N-1}|q(x+n\omega)q(x+(n+1)\omega)|^{\frac{1}{2}}}.\nonumber\end{eqnarray}
One obtains
\begin{eqnarray}\label{41003}
\|T_N(x,E)-T_N(x,E')\|&\leq&\sum_{j=0}^{N-1}(\|B(x+(N-1)\omega,E)\times\cdots\times
B(x+(j+1)\omega,E)\|\times \\&\ &\ \ \ \ \ \ \ \ \ \ \
\|B(x+j\omega,E)-
B(x+j\omega,E')\|\times\|B(x+(j-1)\omega,E')\times\cdots\times
B(x,E')\|)
\nonumber\\
&=&\sum_{j=0}^{N-1}\|\prod_{m=1}^{N-j}B(x+(N-m)\omega,E)\|\times
\|\prod_{m=j-1}^0B(x+m\omega,E')\|\times|E-E'|\nonumber.
\end{eqnarray}
Combining (\ref{41002}), (\ref{41003}) and Lemma \ref{22020}, then
\begin{eqnarray}
(\ref{41002})&\leq&\sum_{j=0}^{N-1}\|\prod_{m=1}^{N-j}B(x+(N-m)\omega,E)\|\times
\|\prod_{m=j-1}^0B(x+m\omega,E')\|\times \exp(-ND)\times|E-E'|\\
&&\times\exp\left (-\sum_{n=o}^{N-1}\frac{1}{2}\log|q(x+n\omega)q(x+(n+1)\omega)|+ND\right )\nonumber\\
 &=&\sum_{j=0}^{N-1}\|\prod_{m=1}^{N-j}B(x+(N-m)\omega,E)\|\times
\|\prod_{m=j-1}^0B(x+m\omega,E')\|\times \exp(-ND)\times \exp(-N
F_N(x))\times |E-E'| \nonumber
\end{eqnarray}
Combining (\ref{60036}) with (\ref{60038}), one has
\begin{equation}L_N(E)\leq
L(E)+\frac{1}{25}L(E_0)<\frac{11}{10}L(E_0)+\frac{1}{25}L(E_0)=\frac{57}{50}L(E_0)
\end{equation} for any $N\geq \tilde{N}_0$ and any
$E\in(E_0-\rho'_0,E_0+\rho'_0)$, where $\tilde{N}_0$ and $\rho'_0$
are as in Lemma \ref{62006}. Let $N>N_2:=2\frac{\log
C(p,q)-D}{L(E_0)}\tilde{N}_0\geq 2\tilde{N}_0$ and $E,E'\in
(E_0-\rho'_0, E_0+\rho'_0)$, where $(\log (p,q)-D)$ is as in Remark
\ref{23023}, $\tilde{N}_0$ and $\rho'_0$ are as in Lemma
\ref{62006}. Due to Lemma \ref{22020} and Remark \ref{23023}, one
has
\begin{eqnarray}\\
 (\ref{41002})&\leq & \left (\sum_{j=1}^{\tilde{N}_0}+\sum_{j=\tilde{N}_0+1}^{N-\tilde{N}_0}
 +\sum_{j=N-\tilde{N}_0+1}^{N}\right ) \|\prod_{m=1}^{N-j}B(x+(N-m)\omega,E)\|\times
\|\prod_{m=j-1}^0B(x+m\omega,E')\|\times \exp(-ND)\nonumber\\
&&\ \ \ \ \ \ \ \ \ \times \exp(-N F_N(x))\times |E-E'|\nonumber\\
&\leq &  \left (2\sum_{j=1}^{\tilde{N}_0} \exp\left ((\log
C(p,q)-D)\tilde{N}_0+\frac{57}{50}L(E_0)N+C_6(\frac{\log
N}{N})^{\frac{1}{2}}-D
\right)+\sum_{j=\tilde{N}_0+1}^{N-\tilde{N}_0} \exp\left
(\frac{57}{50}L(E_0)N+2C_6(\frac{\log
N}{N})^{\frac{1}{2}}-D \right)\right )\nonumber\\
&&\ \ \ \ \ \ \ \ \ \times \exp(-N F_N(x))\times |E-E'|\nonumber\\
&\leq&\sum_{j=1}^N \exp\left
(\frac{1}{2}L(E_0)N+\frac{57}{50}L(E_0)N+2C_6(\frac{\log
N}{N})^{\frac{1}{2}}-D \right)\times \exp(-N F_N(x))\times
|E-E'|\nonumber.\end{eqnarray} There exist $N_3=N_3(p,q,\omega,E_0)$
s.t. for any $N\geq N_3$ holds
\[\sum_{j=1}^N \exp\left
(\frac{82}{50}L(E_0)N+2C_6(\frac{\log N}{N})^{\frac{1}{2}}-D
\right)\leq \exp(2L(E_0)N).\]It implies that for any $N\geq
\max(N_2,N_3)$ holds \[ (\ref{41002}) \leq
\exp(2L(E_0)N)\times\exp(-N
 F_N(x))\times|E-E'|.\]It is obvious that  \[ (\ref{41002}) \leq
\exp(2L(E_0)N)\times\exp(-N
 F_N(x))\times|E-E'|,\] when $\|\tilde{M}_N(x,E,\omega)\|\geq
\|\tilde{M}_N(x,E'\omega)\|$. Set
\[\mathbb{B}:=\{x:NF_N(x)<-NL(E_0)\},\]
then
\[\mes(\mathbb{B})\leq
\mes(\{x:|NF_N(x)-N<F_N>|>NL(E_0)\})<\exp(-cL(E_0)N),\] since
$<F_N>=0$. It implies
\[(\ref{41002})<\exp(3L(E_0)N)|E-E'|,\mbox{ if }
x\in\mathbb{B}.\] Due to (\ref{11111}), one has
\begin{eqnarray}
|L_N(E)-L_N(E')|&=&
\int_{\mathbb{T}\backslash\mathbb{B}}|\tilde{u}_N(x,E)-\tilde{u}_N(x,E')|dx+\int_{\mathbb{B}}|\tilde{u}_N(x,E)-\tilde{u}_N(x,E')|dx\\
&<& \exp
(3L(E_0)N)|E-E'|+2\tilde{C}(p,q)^{\frac{1}{2}}\exp(-\frac{c}{2}L(E_0)N)\nonumber.
\end{eqnarray}
Let $N\geq N_4:=\max(N_1,N_2,N_3)$, where $N_1$ is as in Lemma
\ref{22011}. Due to Lemma \ref{62010}
\begin{eqnarray}|L(E)-L(E')|&\leq &
|L(E)+L_N(E)-2L_{2N}(E)|+|L(E')+L_N(E')-2L_{2N}(E')|\\
&\ &\ \ +|L_N(E)-L_N(E')|+2|L_{2N}(E)-L_{2N}(E')|\nonumber\\
&< &2\exp(-c_6L(E_0)N)+3\exp(6L(E_0)N)|E-E'|+3\tilde{C}(p,q)^{\frac{1}{2}}\exp(-\frac{c}{2}L(E_0)N)\nonumber\\
&< &\exp(-c_7L(E_0)N)+3 \exp (6L(E_0)N)|E-E'|,
\nonumber\end{eqnarray} where $c_7=c_7(p,q,\omega)$. Set
$\rho''_0=\exp(-(6+c_7)L(E_0)N_4)$, and $\rho_0=\min(\rho'_0,
\frac{\rho''_0}{2}).$ Then for $E,E'\in (E_0-\rho_0, E_0+\rho_0)$,
there exists $N\geq N_4$ such that
\[\exp\big(-(6+c_7)L(E_0)(N+1)\big)\leq |E-E'|\leq \exp\big(-(6+c_7)L(E_0)N\big).\]
It implies
\begin{eqnarray}|L(E)-L(E')|&<&4\exp(-c_7L(E_0)N)=4\exp\big(-\frac{N}{N+1}c_7L(E_0)(N+1)\big)<4\exp\big(-\frac{2c_7}{3}c_7L(E_0)(N+1)\big)\\
&<&\exp(-\frac{c_7}{2}L(E_0)N)<
|E-E'|^{\beta},\nonumber\end{eqnarray} where
$\beta=\frac{c_7}{12+2c_7}$, only depending on $p$, $q$ and
$\omega$. By (\ref{10007}), one also has
\[|J(E)-J(E')|=|L(E)-L(E')|<|E-E'|^{\beta}.\]
\end{proof}

\end{document}